\documentclass[headsepline=true]{scrartcl}
\RequirePackage{etex}  
\usepackage{amsmath}
\usepackage{amsthm, amssymb, amsfonts,bbm}
\usepackage[top=1.15in,bottom=1.15in,left=1in,right=1in]{geometry}
\usepackage{mathtools}
\usepackage[pdfborder={0 0 0}]{hyperref}
\usepackage{tikz}
\usetikzlibrary{decorations.markings}
\usetikzlibrary{decorations.pathreplacing}
\usepackage{enumerate}
\usepackage{subcaption}
\usepackage[capitalize]{cleveref}
\usepackage{autonum}
\usepackage{wrapfig}
\usepackage[font=small]{caption}

\title{On the zeros of odd weight Eisenstein series}
\author{Jan-Willem van Ittersum%
\thanks{\emph{Email}: \href{mailto:j.w.ittersum@uni-koeln.de}{j.w.ittersum@uni-koeln.de}, \newline
Max-Planck-Institut f\"ur Mathematik, Vivatsgasse 7, 53111 Bonn, Germany. 
\emph{Current address:}
University of Cologne, Department of Mathematics and Computer Science, Weyertal 86-90, 50931
Cologne, Germany
}
, Berend Ringeling%
\thanks{\emph{Email}: \href{mailto:b.ringeling@math.ru.nl}{b.ringeling@math.ru.nl}, \newline
Department of Mathematics, IMAPP, Radboud University, PO Box 9010, 6500~GL Nijmegen, The Netherlands
}
}

\newcommand{\sltwoz}{\mathrm{SL}_2(\z)}

\DeclareMathOperator{\sgn}{sgn}

\newcommand{\GEis}{\mathbb{G}}
\newcommand{\FD}{\overline{\mathcal{F}}}
\newcommand{\HH}{\mathfrak{H}}

\newcommand{\ii}{\mathrm{i}}
\renewcommand{\Re}{\mathrm{Re}\,}
\renewcommand{\Im}{\mathrm{Im}\,}
\newcommand{\round}{\mathrm{Round}}
\newcommand{\VOA}{\mathrm{VOA}}
\newcommand{\abcd}{\left(\begin{smallmatrix} a & b \\ c & d \end{smallmatrix}\right) }

\renewcommand{\=}{\: =\: }
\newcommand{\defis}{\: :=\: }
\newcommand{\+}{\,+\,}
\newcommand{\meno}{\,-\,}

\renewcommand{\phi}{\varphi}
\definecolor{darkgreen}{rgb}{0.0, 0.5, 0.0}
\newcommand{\n}{\mathbb{N}}
\newcommand{\z}{\mathbb{Z}}
\newcommand{\q}{\mathbb{Q}}
\renewcommand{\r}{\mathbb{R}}
\renewcommand{\c}{\mathbb{C}}

\newcommand{\dd}{\mathop{}\!\mathrm{d}}
\newcommand{\pdv}[2]{\frac{\partial #1}{\partial #2}}

\theoremstyle{plain} 
\newtheorem{thm}{Theorem}[section]
\newtheorem{lem}[thm]{Lemma} 
\newtheorem{cor}[thm]{Corollary} 
\newtheorem{prop}[thm]{Proposition}

\theoremstyle{definition}

\theoremstyle{remark}
\newtheorem*{remarkx}{Remark}
\newenvironment{remark}
  {\pushQED{\qed}\remarkx}
  {\popQED\endremarkx}

\begin{document}
\maketitle
\begin{abstract}
We count the number of zeros of the holomorphic odd weight Eisenstein series in all $\mathrm{SL}_2(\mathbb{Z})$-translates of the standard fundamental domain.
\end{abstract}

\section{Introduction}
Let $\tau \in \HH$, the complex upper half plane. In a famous work, Fanny Rankin and Peter Swinnerton-Dyer showed that all the zeros of the Eisenstein series 
\begin{align}\label{eq:Ek}
E_k(\tau) \= \frac{1}{2}\sum_{\substack{m,n\in \z \\(m,n)=1}} (m\tau+n)^{-k} \qquad (k\geq 4 \text{ even}, (m,n):=\gcd(m,n))
\end{align}
for the full modular group~$\Gamma=\sltwoz$ lie on $\Gamma$-translates of the unit circle \cite{RSD70}. The main idea of their (only one-page) proof is that $e^{\ii k\theta/2}E_k(e^{\ii \theta})$ is a real-valued function for $\theta\in (\pi/3,2\pi/3)$ and that this function is well-approximated by a cosine, i.e.,
\[e^{\ii k\theta/2}E_k(e^{\ii \theta})\=
2\cos k\theta/2 \+ R(\theta) 
.\] The result follows as the weighted number of zeros of~$E_k$ in the standard fundamental domain is $\frac{k}{12}$ (by the modularity of~$E_k\mspace{1mu}$; see \cref{sec:fd}), the cosine has a corresponding number of sign changes and the remainder~$R$ satisfies $|R|<2$. 

For $k=2$ the above sum~\eqref{eq:Ek} does not converge absolutely. However, one can extend the definition of the Eisenstein series by the Eisenstein summation procedure, or, equivalently, by the $q$-expansion
\[ E_k(\tau) \defis 1 \+ c_k\sum_{n\geq 1} \sigma_{k-1}(n)\,q^{n} \qquad (k\geq 2, q=e^{2\pi \ii \tau}),\]
where $\sigma_{k-1}(n) = \sum_{d\mid n} d^{k-1}$ is the well known divisor sum and $c_k$ is the constant $c_k=\frac{(-2\pi \ii)^k}{\zeta(k)(k-1)!}$.
Note that this $q$-expansion is non-trivial, also for odd $k$. Hence, as a byproduct, we now have attained a definition of the main object of study in this work, that is, the Eisenstein series of \emph{odd} weight~$k$. In contrast to the even weight Eisenstein series, which for $k\geq 4$ is modular and for $k=2$ is quasimodular, the odd weight Eisenstein series are not (quasi)modular. The odd weight Eisenstein series are \emph{holomorphic quantum modular forms}, a much weaker notion recently defined by Zagier \cite{Zag20, Whe23}. 

Another, more intrinsic, definition of the even and odd weight Eisenstein series is as follows \cite{GKZ06}. Let 
$G_k$ be given by
\begin{equation}\label{eq:Gk} G_k(\tau) \defis  \sum_{\mu \succ 0}\hspace{-2pt}\mbox{}^{\raisebox{3pt}{\scriptsize{e}}}\hspace{2pt}\frac{1}{\mu^k} \qquad (k\geq 2),
\end{equation}
where $\mu=m\tau+n\in \z\tau+\z$ and the total order~$\succ$ on $\z\tau+\z$ is given by $\mu \succ 0$ if $ m>0$ or if $m=0$ and $n > 0$, and $\mu\succ \nu$ if $\mu-\nu\succ 0$. 
In case $k=2$, the sum does not converge absolutely, and we apply the Eisenstein summation procedure $\sum_{\mu \succ 0}^{\text{e}}:=  \sum_{m=0,n>0}\+\sum_{m>0}\sum_{n\in \z}$. 
Then, 
\[G_k\=\zeta(k)\, E_k \qquad (k\geq 2).\] 

If $E_k$ is not a modular form, there seems a priori neither a reason for an interesting distribution of its zeros nor machinery to count these zeros. Namely, observe that for $k=2$ and odd~$k$, the zeros of~$E_k$ are not invariant under the modular group~$\Gamma$, nor is the number of zeros independent of the choice of a fundamental domain. To our surprise, both concerns can be overcome. 
For the quasimodular Eisenstein series $E_2\mspace{1mu}$, two groups of authors independently determined the distribution of its zeros \cite{IJT14, WY14}, namely, the centers of the Ford circles form a high-precision approximation for the location of these zeros. Both works build on a tool, developed in different works of Sebbar (e.g., \cite{ES10}), which then later was used to determine the distribution of the zeros of derivatives of all even weight Eisenstein series in \cite{GO22} and of quasimodular forms by the authors of the present paper~\cite{IR22b}. 


In this paper, we show how to use the ideas of Rankin and Swinnerton-Dyer to determine the distribution of zeros of the odd weight Eisenstein series. These ideas have been applied in many works on zeros of modular forms, among which in \cite{Ran82} to certain Poincaré series, and in \cite{RVY17} to show that cusp forms of the form $E_k E_\ell-E_{k+\ell}$ (with $k,\ell\geq 4$ even and sufficiently large) have all zeros on the boundary of the fundamental domain. By using these ideas, we bypass the tool of Sebbar, which is not available for the non-quasimodular odd-weight Eisenstein series. 

Write $N_\lambda(f)$ for the weighted number of zeros of~$f$ in~$\gamma\FD$, where $\lambda$  is related to $\gamma=\abcd\in \Gamma$ by $\lambda(\gamma)=\lambda=-\frac{d}{c}\in \mathbb{P}^1(\q)$, and $\FD$ is the closure of the standard fundamental domain for $\Gamma=\sltwoz$ (see \cref{sec:fd}). Recall $N_\lambda(E_k)=\frac{k}{12}$ for \emph{even}~$k$. Now, for \emph{odd}~$k$, the number~$\frac{k}{12}$ is a good approximation for the number of zeros of~$E_k$ within some fundamental domain; more precisely, either rounding $\frac{k}{12}$ up or rounding it down, gives the exact number of zeros: 

\begin{thm}\label{thm:1} For all odd $k\geq 3$ and all $\lambda\in \mathbb{P}^1(\q)$ we have
\[ \left|N_\lambda(E_k) -\frac{k}{12}\right|\,\leq\, \frac{3}{4}.\]
More precisely, the value $N_\lambda(E_k)$, depending on $k \pmod{12}$ and $|\lambda|$ can be found in \cref{tab:1}.
 \begin{table}[ht!]
 \arraycolsep=5pt\def\arraystretch{1.4} \begin{center}
 $\begin{array}{r|ccc}
k \mod 12
 &    |\lambda|\leq \frac{1}{2}, & \frac{1}{2}<|\lambda|\leq 1, & |\lambda|>1 \\\hline
1 & \color{blue} \lfloor \frac{k}{12} \rfloor & \color{blue} \lfloor \frac{k}{12} \rfloor & \color{blue}\lfloor \frac{k}{12} \rfloor \\ 
3 & \color{red} \lceil \frac{k}{12} \rceil & \color{blue}\lfloor \frac{k}{12} \rfloor & \color{blue}\lfloor \frac{k}{12} \rfloor \\ 
5 & \color{red} \lceil \frac{k}{12} \rceil & \color{red} \lceil \frac{k}{12} \rceil & \color{blue}\lfloor \frac{k}{12} \rfloor \\ 
7 & \color{blue}\lfloor \frac{k}{12} \rfloor & \color{blue}\lfloor \frac{k}{12} \rfloor & \color{red} \lceil \frac{k}{12} \rceil \\ 
9  &\color{blue} \lfloor \frac{k}{12} \rfloor & \color{red} \lceil \frac{k}{12} \rceil & \color{red} \lceil \frac{k}{12} \rceil \\ 
11& \color{red} \lceil \frac{k}{12} \rceil & \color{red} \lceil \frac{k}{12} \rceil & \color{red} \lceil \frac{k}{12} \rceil 
 \end{array}$
 \end{center}\vspace{-8pt}
  \caption{The value of~$N_\lambda(E_k)$.}\label{tab:1}
 \end{table}
 \end{thm}
Inspired by Rankin and Swinnerton-Dyer, for the standard fundamental domain ($\lambda=\infty$) this result is proven by writing
\[E_k(\tau) 
\= 1\+\frac{1}{\tau^k}\+\frac{1}{(\tau+1)^k}\+\frac{1}{(\tau-1)^k} \+ R_k(\tau)
\qquad(\tau \in \FD),
\]
where the remainder~$R_k$ decreases exponentially as $k\to \infty$ (uniformly in~$\tau$).
It turns out that these four terms $1+\tau^{-k}+(\tau+1)^{-k}+(\tau-1)^{-k}$ determine the distribution of the zeros of~$E_k$ in~$\FD$.
Similarly, we obtain a suitable approximation for $E_k$ in~$\gamma\FD$, where the approximation depends on $\gamma\in \sltwoz$. 

Note that for odd~$k$ the function $e^{\ii k\theta/2}E_k(e^{\ii \theta})$ is no longer real-valued for real~$\theta$. Accordingly, the zeros of~$E_k$ in~$\FD$ for odd $k$ do not lie on the unit circle. 
      In this case, all zeros lie arbitrarily close to the unit circle (as $k\to \infty$).
\begin{thm} \label{thm:2}
For all odd $k\geq 3$ all zeros $z$ of~$E_k$ in~$\FD$ satisfy
\[1<|z|<4^{\frac{1}{k}}.\]
\end{thm}

For even $k$, the Eisenstein series~$E_k$ equals up to a multiplicative constant the series 
\begin{align} \GEis_k(\tau) &\defis -\frac{B_{k}}{2k} + \sum_{m,r \geq 1} m^{k - 1} q^{m r} \, \qquad (k\geq 1, B_k \text{ is the $k$th Bernoulli number}).
\end{align}
This is a consequence of the fact that the even zeta values are given by $\zeta(k) = \frac{B_k}{2k}\frac{(-2\pi \ii)^k}{(k-1)!}$. For odd values of~$k$, this formula is false; even more, it is expected that all odd zeta values are algebraically independent of each other and of~$\pi$. Still, $\GEis_k$ is a well-defined holomorphic function for all $k\geq 1$. Recall $B_k=0$ for $k\geq 3$ odd. Hence, $\GEis_k$ equals (up to a multiplicative constant) the lattice sum
\begin{equation}\label{eq:GGk}\sum_{\substack{\mu\gg 0\\(\mu)=1}} \frac{1}{\mu^k}  \= E_k(\tau)-1 \=  \frac{(-2\pi \ii)^k}{\zeta(k)\,(k-1)!}\GEis_k(\tau) \qquad (k\geq 3 \text{ odd}),\end{equation}
where $\mu=m\tau+n\in \z\tau+\z$, $(\mu):=\gcd(m,n)$ and the partial order~$\gg$ on $\z\tau+\z$ is given by $\mu \gg 0$ if $ m>0$ and $\mu\gg \nu$ if $\mu-\nu\gg 0$. The distribution of the zeros of~$\GEis_k$ in~$\FD$ ($k$ odd) is reminiscent of those of~$E_\ell'$ ($\ell$ even), both of which admit their zeros on the sides $z=\pm \frac{1}{2}$ of the fundamental domain. In contrast, if $\gamma$ does not fix $\ii\infty$,
the series $\GEis_k$ and $E_\ell'$ have a completely different distribution of zeros in~$\gamma \FD$. In fact,  $\GEis_k$ has exactly the same number of zeros in~$\gamma \FD$ as $E_k\mspace{1mu}$, unless $\lambda(\gamma)\in \{0,\pm1,\infty\}$.
\begin{thm}\label{thm:3} For all odd $k\geq 3$ 
and $\lambda\in \mathbb{P}^1(\q)\backslash\{0,\pm 1,\infty\}$ we have
\[ N_\lambda(\GEis_k)\=N_\lambda(E_k).\]
\end{thm}
\begin{thm}\label{thm:4}
The weighted number of zeros of~$\GEis_k$ ($k\geq 3$~odd) in~$\FD$ equals
\[N_\infty(\GEis_k) \= \begin{cases} \bigl\lceil \frac{k}{6} \bigr\rceil & \text{if } k\equiv 3,5,11 \bmod 12\\[3pt]
\bigl\lfloor \frac{k}{6} \bigr\rfloor & \text{if } k\equiv 1,7,9 \bmod 12.
\end{cases}
\]
 All these zeros are located on the vertical boundaries \[\bigl\{\pm \tfrac{1}{2} + \ii t \,:\, t \geq \tfrac{1}{2}\sqrt{3} \bigr\} \,\cup\, \{ \ii \infty\}.\]
\end{thm}
\begin{figure}[!t] 
\caption{\small Zeros of the Eisenstein series~$E_{23}$ and~$\GEis_{23}$ in fundamental domains~$\gamma\FD$ with ${\lambda(\gamma)=0,\pm \frac{1}{2}, \pm {1},\infty}$. Note that the zero of~$\GEis_{23}$ at the cusp at infinity is not displayed. Moreover, on these scales, one cannot observe that the zeros of $E_{23}$ are very close but not on the unit circle, whereas $\GEis_{23}$ admits zeros (exactly) on the side of $\FD$.
}\vspace{7pt}
  \begin{subfigure}[b]{0.49\textwidth}
    \includegraphics[width=\textwidth]{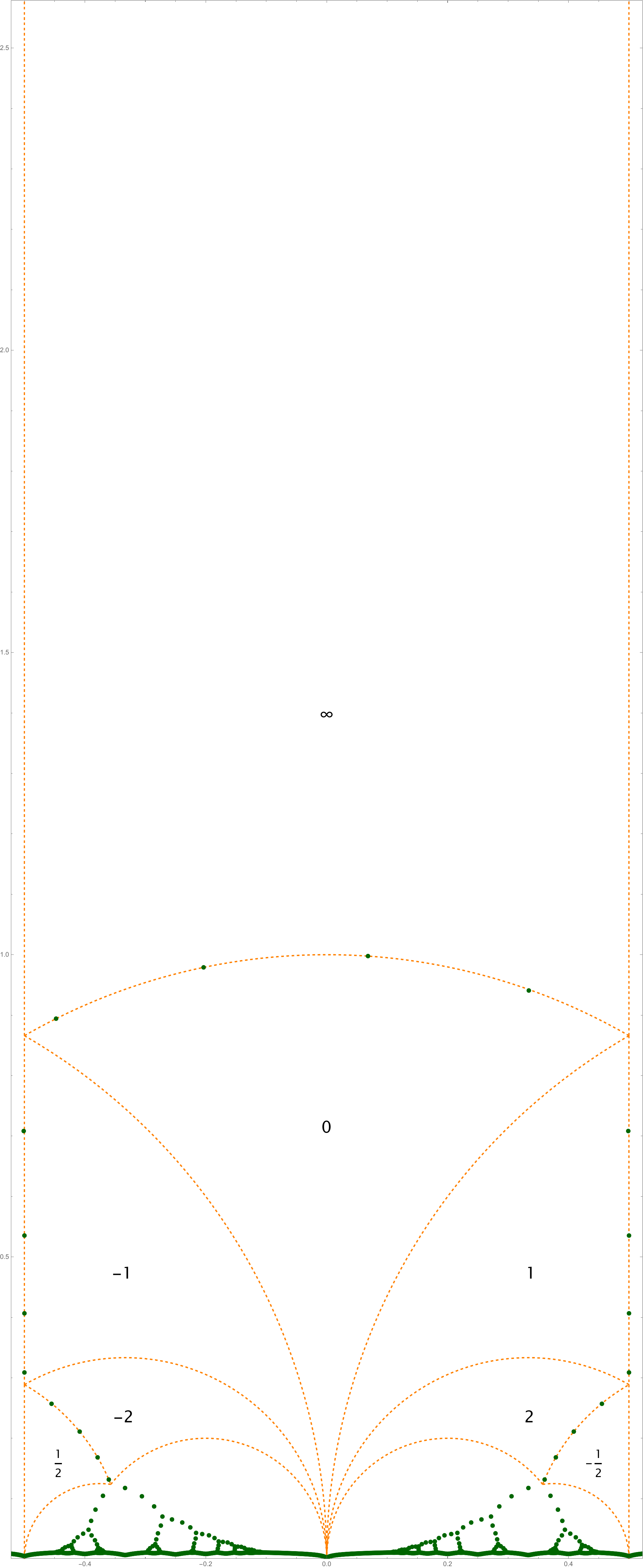}
    \caption{The approximate location of 1500 zeros of~$E_{23}\mspace{1mu}$.}
    \label{fig:zerosa}
  \end{subfigure}
  \hfill
  \begin{subfigure}[b]{0.49\textwidth}
    \includegraphics[width=\textwidth]{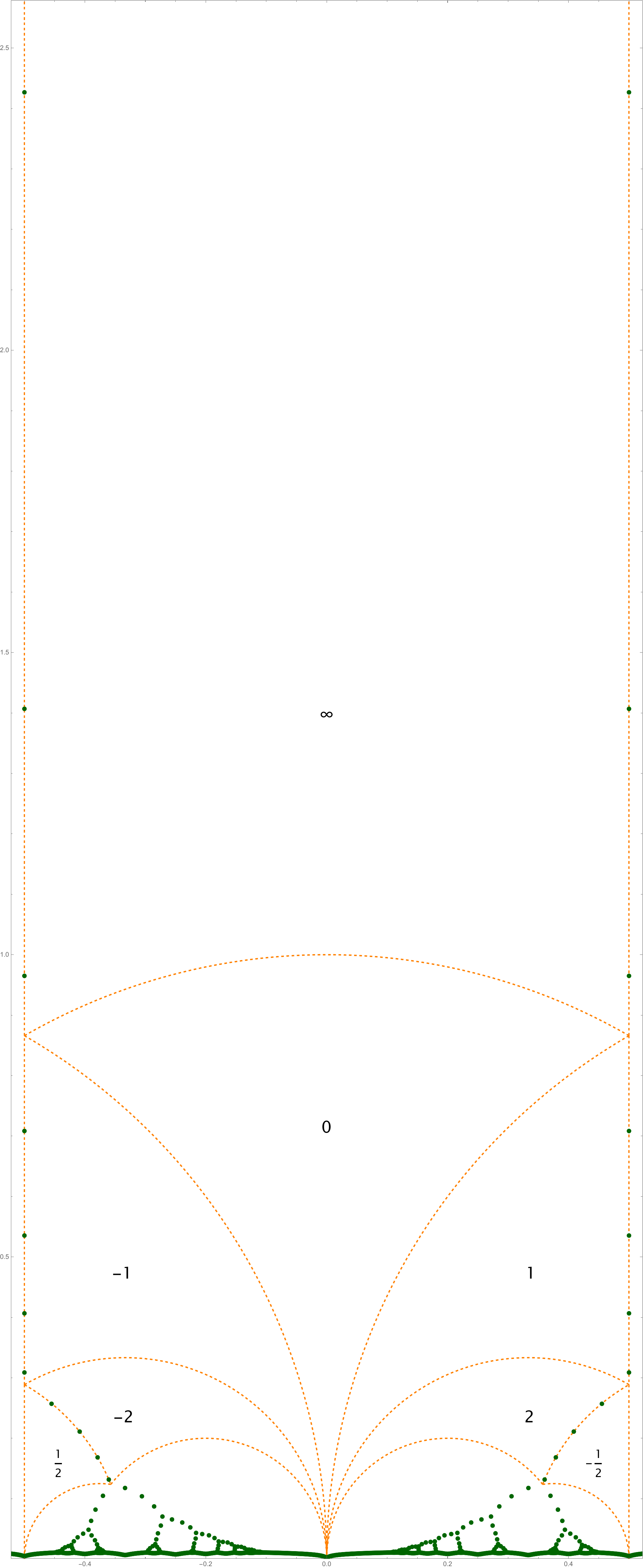}
    \caption{The approximate location of 1500 zeros of~$\GEis_{23}\mspace{1mu}$.}
    \label{fig:zerosb}
  \end{subfigure}
  \label{fig:zeros}
\end{figure}
\begin{remark}\mbox{}\\[-20pt]
\begin{enumerate}[{\upshape (i)}]\itemsep0pt
    \item Let $\gamma \in \Gamma$. Unless $\lambda(\gamma)\in \{\pm 1, \infty\}$, the zeros of~$\GEis_k$ are located in the interior of~$\gamma\FD$. Note that $\lambda(\gamma)=\infty$ corresponds to the zeros of~$\GEis_k$ in~$\FD$ in \cref{thm:4} above. If $\lambda(\gamma)=\pm 1$, the zeros of~$\GEis_k$ lie on~$\gamma \mathcal{C}$ with $\mathcal{C}$ the intersection of~$\partial\FD$ with the unit circle~(defined in~\cref{sec:fd}).
   \item The zeros of~$E_k$ and $E_{k+12}$ are known to interlace on the unit circle for angles $\theta$ with $\frac{1}{2}\pi<\theta<\frac{2}{3}\pi$ if $k$ is even \cite{Gek01,Noz08}. Numerical computations suggest the same interlacing property holds for the arguments of the zeros of~$E_k$ and $E_{k+12}$ in~$\FD$. Moreover, the zeros of~$\GEis_k$ and $\GEis_{k+6}$ in~$\FD$ seem to interlace on the vertical boundaries. 
   \item Upon agreeing that $c_1 = \frac{-2\pi \ii}{\zeta(1)}=0$, i.e., $E_1 \equiv  1$, \cref{thm:1} and \cref{thm:2} trivially extend to $k=1$. In contrast, \cref{thm:3} and \cref{thm:4} are false for~$\GEis_1\mspace{1mu}$; it remains a non-trivial interesting open question to describe the distribution of its zeros. Based on our numerical experiments, it is not clear whether $N_\lambda(\GEis_1)$ is piecewise constant in $\lambda$. Also, note that some authors consider $\GEis_1$ with the different convention for the Bernoulli number~$B_1\mspace{1mu}$, i.e., they define $\GEis_1$ with constant term $-\frac{1}{4}$ rather than $\frac{1}{4}$. \qedhere
\end{enumerate}
\end{remark}

In \cref{sec:infty} we define the counting function~$N_\lambda$ and prove all results on zeros in~$\overline{\mathcal{F}}$, i.e., \cref{thm:2} and \cref{thm:4}. In \cref{sec:alllambda} we extend these results to all $\Gamma$-translates of the standard fundamental domain and prove \cref{thm:1} and \cref{thm:3}. 

\subsection*{Acknowledgement}
A question of Can Turan during the conference \emph{Modular Forms in Number Theory and Beyond} in Bielefeld triggered us to investigate the zeros of the odd weight Eisenstein series. Therefore, we thank him and the organizers of this conference. We also thank Wadim Zudilin for inspiring remarks. The first author was supported by the SFB/TRR 191 “Symplectic Structure in Geometry, Algebra and Dynamics”, funded by the DFG (Projektnr. 281071066 TRR 191).

\section{Zeros in the standard fundamental domain}\label{sec:infty}
\subsection{Preliminaries}
\paragraph{The fundamental domain}\label{sec:fd}
Let $\HH=\{z\in \c \mid \mathrm{Im}(z)>0\}$ be the \emph{complex upper half plane}, $\HH^*=\HH\cup \mathbb{P}^1(\q)$ be the \emph{extended upper half plane} and
 \[\FD \defis \{z \in \HH : |z| \geq 1, -\tfrac{1}{2}\leq \Re(z) \leq \tfrac{1}{2}\}  \cup \{\ii\infty\}\]
the standard (closed) \emph{fundamental domain} for the action of $\Gamma=\sltwoz$ on~$\HH^*$, where $\ii\infty$ is the point $[1:0]\in \mathbb{P}^1(\q)$ at infinity. 

\begin{wrapfigure}{r}{0.35\textwidth} 
    \centering \vspace{-18pt}
    \includegraphics[clip, trim=2.5cm 19.5cm 13.8cm 3.5cm]{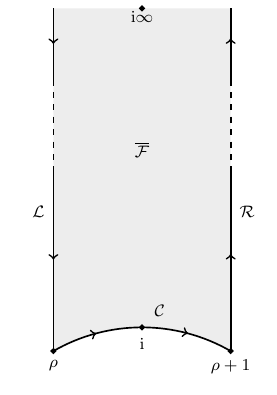}\vspace{-10pt}
    \caption{The fundamental domain}\label{fig:contour}
    \vspace{-70pt}
\end{wrapfigure}

Moreover, we write $\mathcal{C}, \mathcal{L}$ and $\mathcal{R}$ for the circular part, left vertical half-line and right vertical half-line of the boundary~$\partial \FD$ of~$\FD$, i.e., $\partial \FD = \mathcal{L} \cup \mathcal{C} \cup \mathcal{R}$ with
\begin{align}\label{eq:C} \mathcal{C} &\defis \{ z\in \HH : |z|=1, -\tfrac{1}{2}\leq \Re(z)\leq \tfrac{1}{2} \},\\
\label{eq:L} \mathcal{L} &\defis \{ z\in \HH : |z|\geq 1, \Re(z)= -\tfrac{1}{2}\} \,\cup\, \{\ii\infty\},\\
\mathcal{R} &\defis\{ z\in \HH : |z|\geq 1, \Re(z)= \tfrac{1}{2}\} \,\cup\, \{\ii\infty\}.
\end{align}
We orientate these curves such that $\partial \FD$ is positively orientated. 
Write $\rho:=-\frac{1}{2}+\frac{\ii}{2}\sqrt{3}$. To $\tau \in \FD$, we associate the following weight:
\[ 
w(\tau) \defis
\begin{cases}
\frac{1}{6} & \tau \in \{\rho,\rho+1\} \\[2pt]
\frac{1}{2} & \tau \in \partial\FD \backslash\{\rho,\rho+1,\ii\infty\} \\[2pt]
1 & \tau \in (\FD \backslash \partial\FD)\cup\{\ii\infty\}.
\end{cases}
\]
We extend $w$ to $\HH\cup \mathbb{P}^1(\q)$ under the action of~$\Gamma$, i.e.,  $w(\gamma \tau) = w(\tau)$ for all $\gamma \in \Gamma$ and $\tau \in \FD$. 

\paragraph{Order of vanishing at a cusp}
Let $f:\HH\to \c$ be a holomorphic function with some associated weight~$k$. The corresponding \emph{slash action of weight~$k$} is defined by
\begin{align}\label{eq:slash} f|\gamma(\tau) \defis (c\tau+d)^{-k}f(\gamma\tau) \qquad (\gamma=\abcd \in \Gamma),\end{align}
where $\gamma$ acts on~$\HH$ by M\"obius transformations. Note that the element $-1=\left(\begin{smallmatrix} -1 & 0 \\ 0 & -1 \end{smallmatrix}\right)\in \Gamma$ acts trivially on~$\HH$, but $f|(-1)=-f$ is non-trivial for \emph{odd} $k$.

We say $f$ \emph{is holomorphic at infinity} if $f$ admits a Fourier expansion $f=\sum_{n\geq 0} a_n\,q^n$ 
 (with $q=e^{2\pi\ii \tau}$) for some $a_n\in \c$. Moreover, we say $f$ is \emph{well-behaved at a cusp} $\alpha \in \mathbb{P}^1(\q)$, if 
\[\nu_\alpha(f)\defis\lim_{r\to \infty} \int_{-\frac{1}{2}+\ii r}^{\frac{1}{2}+\ii r} \pdv{}{\tau}\log(f|\gamma(\tau)) \dd \tau,\]
where $\gamma\in \Gamma$ is such that $\gamma(\ii\infty) = \alpha$, is a (well-defined) non-negative integer. In that case, we call~$\nu_\alpha(f)$ \emph{the order of vanishing of~$f$ at the cusp} $\alpha$.  Observe that if $f$ is non-trivial and holomorphic at infinity, then 
\begin{align}\label{eq:holatinf} \nu_\infty(f) \= \min\{n\geq 0 \mid a_n \neq 0\}.\end{align}

\begin{lem}
For all odd $k\geq 3$ one has $\nu_\infty(E_k)=0$ and
$\nu_\infty(\GEis_k)=1$. Moreover, for all $\alpha\in \mathbb{P}^1(\q)\backslash\{\infty\}$ one has
\[ \nu_\alpha(E_k) \= 0\= \nu_{\alpha}(\GEis_k).\]
\end{lem}
\begin{proof}
The first part of the statement follows directly from~\eqref{eq:holatinf}. For the other cusps, using the series expansion~\eqref{eq:Gk} and~\eqref{eq:GGk} 
for $f=E_k$ and $f=\GEis_k$ respectively one has for $\gamma \in \Gamma$ (with $\gamma\infty\neq \infty$) that $f|\gamma(\tau)\to \pm 1$ as $\Im\tau \to \infty$, whereas $(f|\gamma)'(\tau)\to 0$. Hence, the odd weight Eisenstein series~$E_k$ does not vanish at a cusp and $\GEis_k$ does not vanish at another cusp than at infinity.
\end{proof}

\paragraph{The counting function}
Let $\gamma\in \Gamma$ be given and $f:\HH\to \c$ be holomorphic and well-behaved at the cusps corresponding to $\gamma$. Write $\lambda(\gamma)=\lambda = \gamma^{-1}(\infty)$. We define the \emph{weighted number of zeros of~$f$ in~$\gamma\FD$} to be
\[ N_\lambda(f) \defis \sum_{\tau \in \gamma\FD} w(\tau)\,\nu_\tau(f) \qquad,\]
where $w(\tau)$ is the weight of~$\tau$ and $\nu_\tau(f)$ the order of vanishing of~$f$ at $\tau$ (extended to the cusps as above). Note that $N_\lambda(f)$ is well-defined, i.e., depends on~$\lambda$ rather than $\gamma$ since $f$ is $1$-periodic.  

For modular forms~$f$ of weight~$k$, such as $f=E_k$ with~$k$ \emph{even}, the \emph{valence formula} states that $N_\lambda(f) = \frac{k}{12}$ for all $\lambda\in \mathbb{P}^1(\q)$. Later, we make use of the following lemma.
\begin{lem}\label{lem:-lambda}
If $f:\HH\to \c$ is holomorphic, well-behaved at the cusps and admitting a Fourier expansion $f= \sum_{n\geq 0} a_n \, q^n$ with \emph{real} Fourier coefficients $a_n\mspace{1mu}$, then
\[ N_\lambda(f) \= N_{-\lambda}(f)\]
for all $\lambda \in \q$. 
\end{lem}
\begin{proof}
Note that 
\[ f(-\overline\tau) \= \sum_{n\geq 0} a_n \, e^{-2\pi i \overline{\tau}} = \overline{f(\tau)}. \]
Hence, $\tau$ is a zero of~$f$ if and only if $-\overline{\tau}$ is a zero of~$f$. 

Now, let $\gamma =\abcd \in \Gamma$ be given and write $\tilde{\gamma} = \left(\begin{smallmatrix} a & -b \\ -c & d \end{smallmatrix}\right)\in \Gamma$. Suppose $\tau \in \gamma \FD$. Then, $\gamma^{-1}\tau\in \FD$. As $\FD$ is invariant under $\tau\mapsto -\overline{\tau}$, this implies $-\overline{\gamma^{-1}\tau}\in \FD$. Now,
\[
-\overline{\gamma^{-1}\tau}=\frac{-d\overline{\tau}+b}{-c\overline{\tau}+a} = \tilde{\gamma}^{-1}(-\overline{\tau}).
\]
Hence, $-\overline{\tau} \in \tilde{\gamma}\FD$. As $\lambda(\gamma)=-\lambda(\tilde\gamma)$, the statement follows. 
\end{proof}

\paragraph{Variation of the argument}
Let $f:\HH\to \c$ be a holomorphic function with some natural weight~$k$. Then, we define 
$\tilde{f}(\tau) \defis \tau^{k/2} f(\tau).$
Often, $|\tau|=1$ and we write
\begin{align}\label{eq:hat}\hat{f}(\theta) \defis \tilde{f}(e^{\ii\theta}) \= e^{k\ii \theta/2} f( e^{\ii\theta}) \qquad \text{for }\theta \in (0, \pi).\end{align}

Suppose $f$ is also holomorphic at infinity and does not admit any zeros on the boundary $\partial \FD$. Then, by a standard application of Cauchy's theorem, the weighted number of zeros of~$f$ in~$\FD$ is given by
\begin{align}\label{eq:VOA1} N_\infty(f) \= \bigl(\VOA_{\mathcal{L}}+\VOA_{\mathcal{C}}+\VOA_{\mathcal{R}}\bigr)(f) \= \frac{k}{12} \+ \bigl(\VOA_{\mathcal{L}}+\VOA_{\mathcal{R}}\bigr)(f)\+\VOA_{\mathcal{C}}(\hat{f}), \end{align}
where the quantity 
\[ \VOA_\mathcal{S}(f)\defis \Re \frac{1}{2 \pi \ii}  \int_\mathcal{S}\frac{f'(\tau)}{f(\tau)} \dd \tau \]
is the \emph{variation of the argument} of~$f$ along some oriented curve~$\mathcal{S}$, and (by a slight abuse of notation), we write
\[ \VOA_\mathcal{C}(\hat{f})\defis -\frac{1}{2 \pi} \Im\!\int_{
\pi/3}^{2\pi/3}\frac{\hat{f}'(\theta)}{\hat{f}(\theta)} \dd \theta \]
for the variation of the argument of~$\hat{f}$. In the sequel, the following observation is crucial. Writing $f(\tau)=r(\tau)\,e^{\ii\alpha(\tau)}$ in polar coordinates with real-analytic radius $r:\HH\to \r_{\geq 0}$ and real-analytic argument $\alpha:\HH\to \r$, one has
\begin{align}\label{eq:VOA} \VOA_{\mathcal{S}}(f) \= \frac{\alpha(s_1)-\alpha(s_0)}{2\pi},\end{align}
where $s_0$ and $s_1$ are the begin- and endpoints of the curve $\mathcal{S}$ respectively. Indeed, $\VOA_{\mathcal{S}}(f)$ is the variation of the argument of~$f$ along~$\mathcal{S}$. 

If $f$ is holomorphic at infinity, then $f$ is $1$-periodic. Hence, $\VOA_{\mathcal{L}}(f)+\VOA_{\mathcal{R}}(f)=0$.
Note that this equation also holds if $f$ admits zeros on~$\mathcal{L}$ or $\mathcal{R}$ (by regularizing the integrals involved as usual by small semi-circles around the roots of~$f$). Hence, we have proven the following statement.
\begin{lem}\label{lem:VOA}  Let $f:\HH\to \c$ be holomorphic, holomorphic at infinity and without zeros on~$\mathcal{C}$. Then,
\begin{align} N_\infty(f) \= \frac{k}{12} \+ \VOA_{\mathcal{C}}(\hat{f}). \end{align}
\end{lem}

\paragraph{Angle estimates}
For $x\in \r$, write $\|x\|:=\min_{n\in \z}|x+2\pi n|$. Often, we make use of the following elementary estimates. \newpage
\begin{lem}\label{lem:angle}
Let $z\in \c$.
\begin{enumerate}[\upshape (i)]\vspace{-5pt}
\item If $z$ is an element of the open ball around $z_0$ with radius $r$, then
\[ 
\bigl\|\arg(z) -\arg(z_0)\bigr\| \,<\, \frac{2r}{|z_0|}\,;
\]
\item If $\Re z>A$ and $|\Im z |<B$ with $A,B\in \r_{>0}\mspace{1mu}$, then
\[ 
\|\arg(z)\| \,<\, \frac{B}{A}.
\]
\end{enumerate}

\end{lem}

\subsection{Zeros of the Eisenstein series~\texorpdfstring{$E_k$}{}}\label{sec:zerosEk}
In this section, we compute the number of zeros of the odd weight Eisenstein series~$E_k$ in the standard fundamental domain~$\FD$, i.e., the number $N_\infty(E_k)$ for $k$ odd. 
Recall that for $\mu=m\tau+n\in \z\tau+\z$, we write $(\mu)=(m,n)$ for the greatest common divisor of~$m$ and~$n$. Also, recall
\[ E_k(\tau) \= \sum_{\substack{\mu \succ 0 \\ (\mu)=1}}\hspace{-3pt}\mbox{}^{\raisebox{3pt}{\scriptsize{e}}}\hspace{3pt} \frac{1}{\mu^k} \= \frac{1}{\zeta(k)} G_k(\tau) \qquad (k\geq 2).\]
From now on, assume that $k\geq 3$ is odd. We write
\begin{align}\label{eq:Ekexp}
{E}_k(\tau) &\= \sum_{\substack{\mu=m\tau+n\\ (m,n) = 1 \\\mu\succ0}} \frac{1}{(m\tau+n)^k}  
\= 1 \+ \frac{1}{\tau^k} \+ \frac{1}{(\tau-1)^k} \+ \frac{1}{(\tau+1)^k} \+ {R}_k(\tau),
\end{align}
where ${R}_k$ contains all terms in the sum~\eqref{eq:Ekexp} for which $m^2+n^2\geq 5$. Then, we have (see~\eqref{eq:hat} for the definition of~$\hat{f}$ for a function~$f$)
\begin{align}\label{eq:Ekhat} \widehat{E}_k(\theta) 
&\= 2 \cos\tfrac{1 }{2}k \theta\+\frac{1}{(2\cos\tfrac{1}{2}\theta)^k} \+ (-1)^{(k+1)/2} \frac{\ii}{(2\sin\tfrac{1}{2}\theta)^k}\+ \widehat{R}_k(\theta).
\end{align}

It was estimated by Rankin--Swinnerton-Dyer that \cite{RSD70}
\begin{align}\label{eq:RSD} |\widehat{R}_k(\theta)| \,\leq \, 4\Bigl(\frac{5}{2}\Bigr)^{-k/2}\+\frac{20\sqrt{2}}{k-3}\Bigl(\frac{9}{2}\Bigr)^{(3-k)/2} \qquad (k>3)\end{align}
In particular, as we will need later, one obtains (we note, for the last time, that $k$ is odd)
\begin{align}\label{eq:Ekrho}
\bigl|E_k(\rho) - \bigl(1-\chi(k)\,\ii\sqrt{3}\bigr)\bigr|  \,\leq \, 
3^{-k/2} + 4\Bigl(\frac{5}{2}\Bigr)^{-k/2}+\frac{20\sqrt{2}}{k-3}\Bigl(\frac{9}{2}\Bigr)^{(3-k)/2} \qquad (k>3)
,
\end{align}
where $\chi:\z \to \{\pm 1\}$ is the primitive Dirichlet character mod~$3$, i.e.,
\[ 
\chi(k) \defis 
\begin{cases} 1 & \textnormal{if } \, k \equiv 1 \bmod 3 \\
-1 & \textnormal{if } \, k \equiv -1 \bmod 3 \\
0 & \text{else.}
\end{cases}
\]

 We proceed in a similar way to estimate $\Im \widehat{R}_k$. \label{page:rkhat} Write
 \begin{align}\label{eq:I} I_k(\tau)  \defis \sum_{\substack{m^2+n^2\geq 5 \\ m>0,n<0 \\ (m,n)=1}} \frac{1}{|m\tau+n|^k} .\end{align}
Observe 
\[\Im \frac{1}{(me^{\ii\theta/2}+ne^{-\ii\theta/2})^k} \+ \Im \frac{1}{(ne^{\ii\theta/2}+me^{-\ii\theta/2})^k}\=0 \qquad (m,n\in \z).\] 
Hence, all terms in~$\Im \widehat{R}_k(\theta)$ with $m,n>0$ cancel in pairs, and we obtain
\begin{align}\label{eq:rkhat}
\Im \widehat{R}_k(\theta) \= \frac{1}{\ii}\sum_{\substack{ m^2+n^2\geq 5 \\ m>0, n<0 \\ (m,n)=1}} \frac{1}{(me^{\ii\theta/2}+ne^{-\ii\theta/2})^k}.
\end{align}
Therefore, 
$|\Im \widehat{R}_k(\theta)| \leq I_k(e^{\ii\theta}).$
\begin{lem}\label{lem:I} For all $k\geq 9$ we have
\begin{align}
I_k(e^{\ii \theta}) &\,\leq\, \frac{4\sqrt{10}}{5^{k/2}}\+\begin{cases} 2\cdot3^{-k/2} & \theta\in [\pi/3,\pi/2] \\ 2\cdot5^{-k/2} & \theta\in [\pi/2,2\pi/3].\end{cases}
\end{align}
\end{lem}
\begin{proof}
 For integers $m,n$ we have
\begin{equation}
\label{estimate}
{|me^{\ii\theta}+n|^2} \= {m^2+n^2 +2mn \cos\theta} \,\geq\, \tfrac{1}{2}(m^2+n^2)
\end{equation}
for $\theta\in [\pi/3,2\pi/3]$. There are at most $N^{1/2}$ integer solutions of the equation $m^2+n^2=N$ when the sign of~$m$ and $n$ are fixed. Picking out the terms with $m^2+n^2=5$ separately, we obtain
\begin{align}
\label{eq:Ikest}
I_k(e^{\ii \theta}) 
&\,\leq \, \frac{2}{|e^{\ii\theta/2}-2e^{-\ii\theta/2}|^k} \+
  \sum_{N\geq 10} \frac{N^{1/2}}{(\frac{1}{2}N)^{k/2}} 
\,\leq\,  \frac{2}{|e^{\ii\theta/2}-2e^{-\ii\theta/2}|^k}\+\frac{4\sqrt{10}}{5^{k/2}},
\end{align}
where in the last step we estimated the sum by an integral under the assumption that $k\geq 9$. 
For the terms with $m^2+n^2=5$, we have more explicitly
\begin{align}
\frac{1}{|e^{\ii\theta/2}-2e^{-\ii\theta/2}|^k} &\= \frac{1}{(5-4\cos(\theta))^{k/2}} 
\,\leq\, \begin{cases} 3^{-k/2} & \theta\in [\pi/3,\pi/2] \\ 5^{-k/2} & \theta\in [\pi/2,2\pi/3].\end{cases} \qedhere
\end{align}
\end{proof}

\begin{lem}\label{lem:thm1}
For $k\geq 11$ the value $\Im \widehat{E}_k(\theta)$ is non-zero for $\frac{\pi}{3}\leq \theta\leq \frac{2\pi}{3}.$
\end{lem}
\begin{proof}
By \cref{lem:I} and~\eqref{eq:Ekhat}, for all odd $k\geq 9$ we have
\begin{align} \Im \widehat{E}_k(\theta) &\= \frac{(-1)^{\frac{k+1}{2}}}{(2\sin \tfrac{1}{2}\theta)^k}\+\Im \widehat{R_k}(\theta) .
\end{align}
Hence,
\begin{align}
\bigl|\Im \widehat{E}_k(\theta)\bigr|&\,\geq \, 
\begin{cases} \displaystyle
{2^{-k/2}}\meno{2}\cdot{3^{-k/2}}\meno 4\sqrt{10}\cdot{5^{-k/2}} & \theta\in [\pi/3,\pi/2] \\ \displaystyle
{3^{-k/2}}\meno (4\sqrt{10}+2)\cdot{5^{-k/2}} & \theta\in [\pi/2,2\pi/3].\end{cases}
\end{align}
For $k \geq 11$ the right-hand side is positive. 
\end{proof}

Let $\round(x)$ be the nearest integer to $x$ (if $x$ is a half-integer, we round up). 
\begin{prop}\label{prop:thm1}
For all odd $k\geq 3$ we have 
\[N_\infty(E_k) \= \round\Bigl(\frac{k}{12}\Bigr) .\]
\end{prop}
\begin{proof}
As the integral $\VOA_{\mathcal{C}}(\widehat{E}_k)$ measures the variation of the argument of~$\widehat{E}_k$ on~$\mathcal{C}$ (see~\eqref{eq:VOA}), the previous lemma implies $|\VOA_{\mathcal{C}}(\widehat{E}_k)|< \frac{1}{2}$. Moreover, it follows from that lemma that $N_\infty(E_k)$ takes an integer value: $E_k$ does not admit a zero on~$\mathcal{C}$, where zeros are weighted by weight $\frac{1}{2}$ or $\frac{1}{6}$ rather than weight~$1$---note that if $E_k$ has a zero $\tau_0$ on~$\mathcal{L}$, counted with weight $\frac{1}{2}$, also $\tau_0+1$ is a zero of~$E_k$ on~$\mathcal{R}$ as $E_k$ is $1$-periodic. 

Now, by \cref{lem:VOA} we obtain $N_\infty(E_k)=\round(\frac{k}{12})$ if $k\geq 11$. Using a more careful error estimate in \cref{lem:I}, one shows that the same holds for $k\in \{3,5,7,9\}$.
\end{proof}

\subsection{Refined estimates for the location of the zeros}
In this section we prove \cref{thm:2}, i.e., we show that all zeros of~$E_k$ in~$\FD$ tend to the unit circle as $k\to \infty$. First of all, we extend the bound~\eqref{eq:RSD} to all $|\tau|\geq 1$:
\begin{lem}\label{lem:Rktau} For $k\geq 11$ and $\tau \in \FD$, we have
\[ |R_k(\tau)|\,\leq \,\frac{6\sqrt{5}}{(5/2)^{k/2}}. \]
\end{lem}
\begin{proof}
For integers $m,n$ we estimate
\begin{equation}
\label{estimate}
{|m\tau+n|^2} \= {m^2r^2+n^2 +2mnr\cos\theta} \,\geq\, \tfrac{1}{2}(m^2+n^2),
\end{equation}
where $\tau=re^{\ii \theta}$ with $r\geq 1$ and $\theta\in [\pi/3,2\pi/3]$. If, $N\geq 5$, there are at most $2N^{1/2}$ coprime integer solutions of the equation $m^2+n^2=N$ with $m>0$. Therefore, 
\begin{equation}
\label{eq:rkest}
|R_k(\tau)|\, \leq \,\!\!\!\sum_{\substack{m>0 \\ m^2 + n^2 \geq 5 \\ (m,n) = 1 }}\frac{1}{|m\tau+n|^k} \,\leq\,  \sum_{N\geq 5} \frac{2N^{1/2}}{(\frac{1}{2}N)^{k/2}}
\,\leq\, \frac{6\sqrt{5}}{(5/2)^{k/2}},
\end{equation}
where in the last step we estimated the sum by an integral under the assumption that $k\geq 11$. 
\end{proof}
In \cref{lem:thm1} we showed $|\Im \widehat{E}_k|\neq 0$. Now, we prove a similar result for a slightly different function---we consider $\Re E_k$ rather than $\Im \widehat{E}_k\mspace{1mu}$. 
\begin{lem}
\label{lem:12withr}
For $k \geq 11, \theta \in (\pi/3, 2 \pi /3)$ and $r\geq 4^{\frac{1}{k}}$ we have
\[ \Re E_k(r e^{\ii \theta}) \,\neq\, 0. \] 
\end{lem}
\begin{proof}
Let $\tau=re^{\ii \theta}$. We have
\[ 1+\Re\frac{1}{\tau^k}
\,\geq\, 1-\frac{1}{r^k}\]
and estimate
\begin{align}\label{eq:2ndtermes}
\label{Estm2}
\left| \Re \, \frac{1}{(\tau+1)^k} \+ \Re \, \frac{1}{(\tau-1)^k}  \right| &\,\leq\, \frac{1}{|\tau+1|^k}\+\frac{1}{|\tau-1|^k} \\
&\= \frac{1}{(r^2+1+2r\cos(\theta))^{k/2}}\+\frac{1}{(r^2+1-2r\cos(\theta))^{k/2}}.
\end{align}
Given $r$, the right-hand side attains its minimum as a function of~$\theta$ on the boundary of $(\pi/3, 2\pi/3)$. Hence, we find
\begin{align}\left| 1+\Re\frac{1}{\tau^k}+\Re\frac{1}{(\tau+1)^{k}} + \Re\frac{1}{(\tau-1)^{k}}  \right| &\geq 1-\frac{1}{r^k}-\frac{1}{(r^2+1+r)^{k/2}}-\frac{1}{(r^2+1-r)^{k/2}} \\
&\geq 1\meno\frac{1}{r^{k}}\meno\frac{1}{3^{k/2}}\meno\frac{1}{r^{k/2}}.
\end{align}
If we assume that
$r \geq 4^{\frac{1}{k}}$ and $k\geq 11$, we find that the right-hand side is at least $ \frac{6\sqrt{5}}{(5/2)^{k/2}},$ so that by \cref{lem:Rktau} we conclude that $\Re E_k(r e^{\ii \theta})\neq 0$. 
\end{proof}
\begin{proof}[Proof of~\cref{thm:2}]
For $k\geq 11$, the statement follows directly from the previous lemma. By a numerical approximation of the root of~$E_k$ in~$\FD$ for $k\in \{7,9,11\}$, we find that the result holds for all $k\geq 3$.
\end{proof}

\begin{remark}
The roots $z$ of~$E_k$ in~$\FD$ satisfy $1<|z|<C^{1/k}$ for $C=4$. Though the constant~$C$ may certainly be improved, we do expect that an upper bound for the radius of the form $C^{1/k}$ is best possible. Note that Dimitrov's theorem (the former Schinzel--Zassenhaus conjecture) provides a lower bound of the same shape $C^{1/k}$ for the \emph{house} of polynomials of degree~$k$, where the house is the maximum of the absolute values of all its roots \cite{Dim19}. Now, the zeros of
\[1+\frac{1}{\tau^k}+\frac{1}{(\tau+1)^k}+\frac{1}{(\tau-1)^k}\]
within~$\FD$ provide a good approximation of the zeros of~$E_k\mspace{1mu}$, which provides some heuristic evidence as to why we believe that a radius of the form $C^{1/k}$ is best possible. 
\end{remark}

\subsection{Upper bound on the number of zeros of~\texorpdfstring{$\GEis_k$}{the Eisenstein series}}
Recall
\[\GEis_k(\tau) \= \frac{(k-1)!}{(-2\pi\ii)^k}\sum_{\mu\gg 0} \frac{1}{\mu^k} \= \frac{(k-1)!}{(-2\pi\ii)^k \zeta(k)} \sum_{\substack{\mu\gg 0 \\ (\mu)=1}} \frac{1}{\mu^k}.\]
Write 
\[H_k(\theta) \defis \ii^{k+1} \Im \, \widehat{E}_k(\theta).\]
Then,
\begin{align} \Im (-1)^{(k+1)/2}\frac{(-2\pi\ii)^k \zeta(k)}{(k-1)!}\widehat{\GEis}_k(\theta) &\= (-1)^{(k+1)/2}\,\Im \, \widehat{E}_k(\theta) \meno (-1)^{(k+1)/2}\,\Im \frac{1}{(0+ 1 \cdot e^{-i\theta/2} )^k} \\
&\=  H_k(\theta) \+ (-1)^{(k-1)/2} \sin(k \theta/2). \label{eq:convexandsine}
\end{align}
To count the number of zeros of~$\GEis_k$ in~$\FD$, we first study this function~$H_k\mspace{1mu}$. 

\begin{lem}
\label{lem:convex}
$H_k$ is a strictly convex positive-valued function on $[\pi/3,2\pi/3]$ for $k\geq 17$.
\end{lem}
\begin{proof}
Note that the proof of \cref{lem:thm1} implies that $H_k$ is positive-valued. Now, a twice-differentiable function is convex if and only if the second derivative is non-negative.
By a similar argument as we used to obtain~\eqref{eq:rkhat}, we find
\[ H_k(\theta) \= \ii^{k}\sum_{\substack{m>0, n<0 \\ (m,n) = 1}}\frac{1}{(me^{\ii\theta/2}+ne^{-\ii\theta/2})^{k}}.\]
We compute the second derivative
\begin{align} H_k''(\theta) 
&\= -\frac{k}{4\ii^k} \sum_{\substack{m>0, n<0 \\ (m,n) = 1}} \frac{4 m n - k (me^{\ii\theta/2}-ne^{-\ii\theta/2})^2}{(me^{\ii\theta/2}+ne^{-\ii\theta/2})^{k+2}} \\
&\= \frac{k}{2^{k+4}}\frac{4+4k\cos(\theta/2)^2}{\sin(\theta/2)^{k+2}} \+ \Im R_k''(\theta).
\end{align}
We treat the remainder similar to \cref{lem:I}. Estimate
\[ 4 mn \leq 2(m^2+n^2) \qquad \textnormal{ and } \qquad |me^{\ii\theta/2}-ne^{-\ii\theta/2}|^2 \leq 2 (m^2+n^2).\]
Then, by taking out the terms with $m^2+n^2=5$ separately, we find
\begin{align} |\Im R_k''(\theta)| 
&\,\leq\,  \frac{k}{2}\frac{k(5+4\cos\theta)^2+8}{(5-4\cos\theta)^{k+2}}\+\frac{k}{4}\sum_{N \geq 10} \frac{(2+2k)N \cdot N^{1/2} }{(\frac{1}{2} N)^{k/2+1}}.
\end{align}
Hence, for $\theta\in (\pi/3,\pi/2)$ we find
\[|\Im R_k''(\theta)|\,\leq\, 
\frac{k}{2}\frac{49k+8}{3^{k+2}} \+
\frac{k(k+1)}{2} \frac{4\sqrt{10}}{5^{k/2}},\]
whereas for $\theta\in(\pi/2,2\pi/3)$ one obtains
\[|\Im R_k''(\theta)|\,\leq\, 
\frac{k}{2}\frac{25k+8}{5^{k+2}} \+
\frac{k(k+1)}{2} \frac{4\sqrt{10}}{5^{k/2}}.\]
Hence, we find
\[ H_k''(\theta) \,\geq\,
    \frac{k(2k+4)}{2^{k/2+2}} \meno \frac{k}{2}\frac{49k+8}{3^{k+2}} \meno
\frac{k(k+1)}{2} \frac{4\sqrt{10}}{5^{k/2}} \]
for $\theta\in (\pi/3,\pi/2)$ and
\[ H_k''(\theta) \,\geq\,   \frac{1}{4}\frac{k(k+4)}{3^{k/2+1}} \meno \frac{k}{2}\frac{25k+8}{5^{k+2}} \meno
\frac{k(k+1)}{2} \frac{4\sqrt{10}}{5^{k/2}} \]
for $\theta\in(\pi/2,2\pi/3)$. Therefore, $ H_k''(\theta) >0$ for $\theta \in (\pi/3, 2\pi/3)$ and $k\geq 17$.
\end{proof}

We now exploit the fact that $\Re \,\mathbb{G}_k(\theta)$ is the sum of a sine and a strictly monotonous function as in~\eqref{eq:convexandsine}, to bound the number of its zeros. 
\begin{lem}\label{lem:zeros}
Let $k\geq 1$ be odd, $A=[\pi/3,2\pi/3]$ and $h:A\to \r_{>0}$ be a continuous strictly convex positive-valued function such that $0<h(\frac{2\pi}{3})<\frac{1}{2}\sqrt{3}$. Then
\[ \sin(k\theta/2) + h(\theta) \quad \text{and} \quad \sin(k\theta/2) - h(\theta)\]
 admit at most
 \[2\Bigl\lfloor \frac{k+5}{12}\Bigr\rfloor+\delta_{k\equiv 5(6)} \qquad \text{and} \qquad 2\Bigl\lfloor \frac{k}{12}\Bigr\rfloor+2-\delta_{k\equiv 1 (6)}\]
 zeros on~$A$ respectively.
\end{lem}
\begin{proof}
Note that the function $s_k:\theta\mapsto \sin(k\theta/2)$ at some point $\theta$ is either positive or convex. Write $\mathcal{I}^-$ and~$\mathcal{I}^+$ for the collection of (closed) intervals on which $s_k$ is non-positive (convex) and non-negative (concave) respectively, where we assume that all intervals~$I$ are maximal with respect to this property. 

Now, recall that a continuous strictly convex function~$f$ admits at most two zeros. Hence, $s_k+h$ admits at most two zeros on each interval~$I\in \mathcal{I}^-$. Now, suppose $k\equiv 5 \bmod 6$. Then, $I=[a,\frac{2\pi}{3})\in \mathcal{I}^-$ for some $a$. As $(s_k+h)(a)=h(a)>0$ and $f(2\pi/3)<\frac{1}{2}\sqrt{3}-h(a)<0$ in that case, there is at most $1$ zero on this interval~$I$. Hence, $s_k+h$ admits at most $2|\mathcal{I}^-|-\delta_{k\equiv 5\,(6)}$ zeros. Similarly, $s_k-h$ admits at most $2|\mathcal{I}^+|-\delta_{\substack{k\equiv 1\,(6)}}$ zeros if $k>1$. To conclude the proof, observe that
\begin{align}
|\mathcal{I}^+|=\Bigl\lfloor \frac{k}{12}\Bigr\rfloor+1, \qquad |\mathcal{I}^-|=\Bigl\lfloor \frac{k+5}{12}\Bigr\rfloor+\delta_{k\equiv 5(6)}\,. & \qedhere
\end{align}
\end{proof}

\begin{lem}
For $k\geq 7$, we have $N_\infty(\GEis_k)\leq \lceil \frac{k}{6} \rceil$. Moreover, if additionally $k$ equals $ 1,7$ or $9 \bmod 12$, we have $N_\infty(\GEis_k)\leq\lfloor \frac{k}{6} \rfloor$.
\end{lem}
\begin{proof}
Assume $k\geq 7$. By~\eqref{eq:Ekrho} we find 
\[\bigl|E_k(\rho)-1 \+ \chi(k)\,\ii\sqrt{3}\bigr|  \,\leq \, 
3^{-k/2} + 4\Bigl(\frac{5}{2}\Bigr)^{-k/2}+\frac{20\sqrt{2}}{k-3}\Bigl(\frac{9}{2}\Bigr)^{(3-k)/2}.\]
Note $\mathcal{G}_k:=E_k-1$ equals $\GEis_k$ up to a multiplicative (imaginary) constant.
Since $\mathcal{G}_k(\rho)\in \r\ii$, we conclude that $\arg \mathcal{G}_k(\rho) = -\pi/2$ if $k\equiv 1 \bmod 3$ and $\arg \mathcal{G}_k(\rho)=\pi/2$ if $k\equiv 2 \bmod 3$. In case $k\equiv 0 \bmod 3$, one needs a more refined estimate. 
Observe
\[ \frac{1}{\rho^k}+\frac{1}{(\rho+1)^k}=0\]
in this case. Now,  
\[ \Bigl|\Im\mathcal{G}_k(\rho) - \frac{1}{(\rho-1)^k}\Bigr| < I_k(2\pi/3). \]
As $|(\rho-1)^{-k}|>I_k(e^{2\pi/3})$ by \cref{lem:I} and using again $\mathcal{G}_k(\rho)\in \r\ii$, we obtain 
\begin{align}\label{eq:signrho} \arg \mathcal{G}_k(\rho) \= \begin{cases}
-\frac{1}{2}\pi & k\equiv 3 \bmod 12 \\ \frac{1}{2}\pi & k \equiv 9 \bmod 12.\end{cases}\end{align}
Hence,
\[
\arg \widehat{\GEis}_k(\rho) \= \begin{cases}
\frac{1}{3}\pi & k\equiv 1 \mod 12\\
0 & k\equiv 3 \mod 12\\
\frac{2}{3}\pi & k\equiv 5 \mod 12\\
-\frac{2}{3}\pi & k\equiv 7 \mod 12\\
0 & k\equiv 9 \mod 12\\
-\frac{1}{3}\pi & k\equiv 11 \mod 12
\end{cases}
\]
and 
\[
\arg \widehat{\GEis}_k(\rho+1) \= \begin{cases}
\frac{1}{6}\pi & k\equiv 1 \mod 6\\
\frac{1}{2}\pi & k\equiv 3 \mod 6\\
-\frac{1}{6}\pi & k\equiv 5 \mod 6.
\end{cases}
\]

Also, observe $H_k(2\pi/3)<\frac{1}{2}\sqrt{3}$  by \cref{lem:I} as
\begin{align} H_k(2\pi/3) \,\leq\, 3^{-k/2}+(4 \sqrt{10}+2)\cdot 5^{-k/2}.\end{align}
This enables us to obtain an upper bound for $\VOA_{\mathcal{C}}(\widehat{\GEis}_k)$  (case by case modulo~$12$) using \cref{lem:zeros} applied to $h = H_k\mspace{1mu}.$ 
The result then follows from~\cref{lem:VOA}.
\end{proof}

\subsection{Lower bound on the number of zeros of~\texorpdfstring{$\GEis_k$}{the Eisenstein series}}
The goal of this section is to show that $\GEis_k$ has all of its zeros on~$\FD$ on the vertical boundaries $\mathcal{L}$ and $\mathcal{R}$. Following \cite[Section~3]{GO22}, we do this by counting the number of sign changes of~$\GEis_k\mspace{1mu}$. 
Define $z_\ell=\frac{1}{2}+\ii t_\ell$ and $t_\ell := \frac{1}{2}\cot(\pi \ell/k)$, where $1\leq \ell \leq \lfloor \frac{k}{6}\rfloor$. Note that for such $\ell$ one has $t_\ell \geq \frac{\sqrt{3}}{2}$. 
Then, \cref{thm:4} will follow from 
\begin{prop} 
\begin{equation}
\label{conjecEis}
    \sgn \GEis_k(z_\ell) = (-1)^\ell .
\end{equation}  
\end{prop}
\begin{proof}
We distinguish two cases, namely, (i)~$3\ell^2\leq k-1$ ($z_\ell$ is `close to $\ii \infty$') and (ii)~$3\ell^2\geq k$ ($z_\ell$ is `close to $\rho$'). 

We start with the first case. Write
\[\GEis_k(\tau) \= \sum_{n = 1}^\infty u_n(\tau), \qquad u_n(\tau) \= n^{k-1}\frac{q^n}{1 - q^n}.\]
It turns out that the dominant contribution of~$\GEis_k(z_\ell)$ is $u_n(z_\ell)$ for $n = \ell$. More precisely, we show that 
\begin{enumerate}[{\upshape (a)}]
    \item $\displaystyle 2|u_{n+1}|<|u_n|$ for $n\geq \ell$;
    \item $\displaystyle 2|u_{n-1}|<|u_n|$ for $2\leq n\leq \ell$.
\end{enumerate}
Write $q_\ell=q|_{\tau=z_\ell}$. Note $e^{-\pi \sqrt{3}}\leq |q_\ell|<1$. 
For part~(a) we observe that \cite[Eqn.~(26)] {GO22}
\[\Bigl(\frac{\ell+1}{\ell}\Bigr)^{k-1}|q_\ell|\leq e^{-3/4}.\]
Hence, if $n\geq \ell$,
\begin{align}
\frac{|u_{n+1}|}{|u_n|} &\= \Bigl(\frac{n+1}{n}\Bigr)^{\!k-1}\,|q_\ell|\,\frac{|1-q_\ell^n|}{|1-q_\ell^{n+1}|} \\
&\,\leq\, \Bigl(\frac{\ell+1}{\ell}\Bigr)^{\!k-1}\,|q_\ell|\,\frac{1+|q_\ell|}{1-|q_\ell|^{2}} \\
&\= \Bigl(\frac{\ell+1}{\ell}\Bigr)^{\!k-1}\,\frac{|q_\ell|}{1-|q_\ell|} \,\leq\, \frac{e^{-3/4}}{(1-e^{-\pi \sqrt{3}})} \,<\, \frac{1}{2}.
\end{align}

For part~(b) we observe that \cite[Eqn.~(29)]{GO22}
\[\Bigl(\frac{\ell-1}{\ell}\Bigr)^{k-1}|q_\ell|\leq e^{-1}.\]
Hence, if $2\leq n\leq \ell$,
\begin{align}
\frac{|u_{n-1}|}{|u_n|} &\= \Bigl(\frac{n-1}{n}\Bigr)^{\!k-1}\,|q_\ell|^{-1}\,\frac{|1-q_\ell^n|}{|1-q_\ell^{n-1}|} \\
&\,\leq\, \Bigl(\frac{\ell-1}{\ell}\Bigr)^{\!k-1}\,|q_\ell|^{-1}\,\frac{1+|q_\ell|^2}{1-|q_\ell|} \,\leq\, e^{-1}\frac{1+e^{-2\pi \sqrt{3}}}{1-e^{-\pi \sqrt{3}}} \,<\, \frac{1}{2}.
\end{align}

Observe $\sgn u_\ell(\frac{1}{2} + \ii t_\ell) = (-1)^\ell$. Writing
\[ \GEis_k = \Bigl(\frac{u_\ell}{2}+u_{\ell-1}\Bigr) + \sum_{j=1}^{\lceil \frac{\ell}{2}-1\rceil}(u_{\ell-2j}+u_{\ell-2j-1}) \+ \Bigl(\frac{u_\ell}{2} + u_{\ell+1}\Bigr) + \sum_{j=1}^{\infty}(u_{\ell+2j}+u_{\ell+2j+1}), \]
where we set $u_{0}=0$, we see that $\GEis_k$ is sum of non-zero terms with sign $(-1)^\ell$. Hence, the statement of the proposition follows in this case. 

Next, we assume that $3\ell^2\leq k$. As we assumed that $1\leq \ell \leq \lfloor \frac{k}{6}\rfloor$, we observe that $k\geq 11$. 
Write
\[ \GEis_k(z_\ell) \= -\frac{(k-1)!}{(2\pi \ii)^k} \sum_{\mu\gg 0} \frac{1}{(m z_\ell+n)^k} \= 
\frac{(k-1)!}{(2\pi)^k}\biggl(2(-1)^\ell |z_\ell|^{-k} \meno \frac{1}{\ii^k} \mathbb{R}_k(z_\ell)\ \biggr)
\]
with 
\[ \mathbb{R}_k(\tau) \defis \sum_{\substack{\mu\gg 0\\ \mu \not \in \{(1,-1),(1,0)\}}} \!\!\!\frac{1}{(m\tau+n)^k}.\]
It remains to show that the remainder $\mathbb{R}_k(z_\ell)$ is in absolute value smaller than the leading contribution. By \cite[{Lemma~14 and~15}]{GO22}, we obtain
\[\frac{|\mathbb{R}_k(t_\ell)|}{|z_\ell|^{-k}} \leq \frac{5}{8} + \sum_{m\geq 2} m^{1-k}\Bigl(\Bigl(\frac{\sqrt{3}}{2}\Bigr)^{-k}+3\Bigr) \,\leq \, 
\frac{5}{8} + (\zeta(10)-1)\Bigl(\Bigl(\frac{\sqrt{3}}{2}\Bigr)^{-11}+3\Bigr) <\frac{2}{3}.
\]
Hence, $\sgn  \GEis_k(z_\ell) =(-1)^\ell.$
\end{proof}
\begin{proof}[Proof of \cref{thm:4}]
The real-valued function $t \mapsto \GEis_k( \frac{1}{2} + \ii t)$ changes sign $\lfloor\frac{k}{6}\rfloor-1$ times for $t\in (\frac{1}{2}\sqrt{3},\infty)$. Hence, additionally including the zero at $\ii \infty$, we find that $\GEis_k$ admits at least $\lfloor \frac{k}{6}\rfloor$ zeros on~$\mathcal{L}$. 
In case $k\equiv 3,5,11 \bmod 12$ we find an additional sign change of $\GEis_k(\frac{1}{2}+\ii t)$ on~$\mathcal{L}$, since the sign at $t=\frac{1}{2}\sqrt{3}$ and $t=t_\ell$ with  $\ell=\lfloor \frac{k}{6}\rfloor$ are different. 
\end{proof}
\begin{cor}\label{cor:thm3infty} For all odd $k$ one has
\[N_\infty(\GEis_k) = \begin{cases} \bigl\lceil \frac{k}{6} \bigr\rceil & \text{if } k\equiv 3,5,11 \bmod 12\\[3pt]
\bigl\lfloor \frac{k}{6} \bigr\rfloor & \text{if } k\equiv 1,7,9 \bmod 12.
\end{cases}
\]
\end{cor}

\section{Zeros in any fundamental domain}\label{sec:alllambda}
\subsection{Preliminaries}
Recall $\Gamma=\sltwoz$. 
For $\gamma=\abcd\in \Gamma$ we have
\begin{align}
{E}_{k}|\gamma(\tau) & \= \sum_{\substack{\mu=m'\tau+n'\\ (m',n') \= 1 \\\mu\succ0}} \frac{1}{((am'+cn')\tau +(bm'+dn'))^k} \\
&\= \frac{1}{(c\tau+d)^k}\+\sum_{\substack{(m,n)=1 \\ dm>cn}} \frac{1}{(m\tau+n)^k}\mspace{1mu}, \qquad \text{where} \qquad \begin{pmatrix} m\\n\end{pmatrix}=\begin{pmatrix} a & c \\ b & d\end{pmatrix}\begin{pmatrix} m'\\n'\end{pmatrix}
\end{align}
and the slash action $|\gamma$ is defined by~\eqref{eq:slash}.

Assume $c<0$ and write $\lambda=-\frac{d}{c}$. If $\lambda\not \in\{-1,0,1,\infty\}$, then
\begin{align}
{E}_{k}|\gamma(\tau)
&\= 1\+\frac{\sgn(\lambda)}{\tau^k}\+\frac{\sgn(\lambda+1)}{(\tau+1)^k}\+\frac{\sgn(\lambda-1)}{(\tau-1)^k} \+ {R}_{k,\lambda}(\tau), \label{eq:Ekhatgamma}
\end{align}
where ${R}_{k,\lambda}$ contains all terms in~${E}_{k}|\gamma$ with $m^2+n^2\geq 5$. 
In the special cases we obtain
\begin{align}
{E}_{k}|\gamma(\tau)
&\= 1\meno\frac{1}{\tau^k}\+\frac{1}{(\tau+1)^k}\meno\frac{1}{(\tau-1)^k} \+ {R}_{k,0}(\tau) \qquad (\lambda=0),\label{eq:Ekhatgamma0}\\
{E}_{k}|\gamma(\tau)
&\= 1\+\frac{1}{\tau^k}\+\frac{1}{(\tau+1)^k}\meno\frac{1}{(\tau-1)^k} \+ {R}_{k,1}(\tau) \qquad (\lambda=1)\label{eq:Ekhatgamma1}.
\end{align}

Similarly,
\begin{align}\label{eq:GkEkcomp}
\mathbb{G}_k|\gamma(\tau) \,\propto\, E_k|\gamma(\tau) \meno \frac{1}{(c\tau+d)^k} \= 1\+\frac{\sgn(\lambda)}{\tau^k}\+\frac{\sgn(\lambda+1)}{(\tau+1)^k}\+\frac{\sgn(\lambda-1)}{(\tau-1)^k} \+ \mathbb{R}_{k,\lambda}(\tau), 
\end{align}
where $\mathbb{R}_{k,\lambda}(\tau) = R_{k,\lambda}(\tau)-\frac{1}{(c\tau+d)^k}$. Note that, with the same proof as in \cref{lem:Rktau}, we have
\begin{equation}\label{eq:rkest2}
|R_{k,\lambda}(\tau)|\,\leq \,\frac{6\sqrt{5}}{(5/2)^{k/2}} \qquad\qquad (k\geq 11 \text{ and } \tau \in \FD). 
\end{equation}

We now set out to prove \cref{thm:1} by explicitly computing $\VOA_{\mathcal{S}}(E_k|\gamma)$ for $\mathcal{S}\in \{\mathcal{L},\mathcal{C},\mathcal{R}\}$ as a function of $\lambda=\gamma^{-1}(\infty)$. As before, to do so, we aim to control the vanishing of the real or imaginary part of~$E_k$ or $\widehat{E}_k$, as we do in a series of lemmas. As we will see, in many cases the same ideas apply to $\GEis_k$ instead of~$E_k$ and $\mathbb{R}_{k,\lambda}$ instead of~$R_{k,\lambda}$ respectively. 

\subsection{The case \texorpdfstring{$\lambda>1$}{lambda bigger than 1}}
\begin{lem} \label{lem:C}
For all $\lambda>0$, the value $\Im \widehat{{E_k}|\gamma}(\theta)$ is non-zero for $\theta\in [\pi/3,2\pi/3].$
\end{lem}
\begin{proof}
As $\lambda>0$, we can assume without loss of generality that $d>0$ and $c<0$. Note that if $m,n\geq 0$ are coprime, then both $(m,n)$ and $(n,m)$ are contained in the domain of summation of the sum in~\eqref{eq:Ekhatgamma}.  
Therefore, 
\[
|\Im\widehat{R}_{k,\lambda}(\theta)| \= \biggl|\frac{1}{(c\tau+d)^k}\+\sum_{\substack{(m,n)=1\\ dm>cn\\ mn<0}} \Im \frac{1}{(me^{\ii\theta/2}+ne^{-\ii\theta/2})^k} \biggr|
\,\leq\, I_k(e^{\ii\theta}), \]
where $I_k$ is defined by~\eqref{eq:I} and the inequality follows from the observation that 
\begin{align}
\{(m,n)\in \z^2 \mid dm>cn \text{ and } mn<0\} \,\cup\, \{(c,d)\} &\,\to\, \{(m,n)\in \z^2 \mid m>0,n<0\} \\
(m,n) &\,\mapsto\, \begin{cases} (m,n) & \text{if } m>0 \\ (-m,-n) & \text{if }  m<0 \end{cases}
\end{align}
 defines a bijection. 
Observe
\begin{align} \Im \widehat{{E_k}|\gamma}(\theta) &\= 
\frac{\epsilon(k,\lambda)}{(2\sin \tfrac{1}{2}\theta)^k}\+\Im \widehat{R}_{k,\lambda}(\theta),
\end{align}
where the sign $\epsilon(k,\lambda)$ is given by $\epsilon(k,\lambda)=\sgn(\lambda-1) (-1)^{(k+1)/2}$. By \cref{lem:I} and the same proof of \cref{lem:thm1} the result follows.
\end{proof}

\begin{lem}\label{lem:R} For all $\lambda>1$ and $k\geq 11$, the value
$\Re E_k|\gamma(\tau)$ is non-zero for $\tau \in \mathcal{R}$. 
\end{lem}
\begin{proof}
Observe that for $\tau \in \mathcal{R}$, we have
\[
\Re E_k|\gamma(\tau) \= 1 \+ \Re\frac{1}{(\tau+1)^k} \+ \Re R_{k,\lambda}(\tau). \]
and
\[|R_{k,\lambda}(\tau)|\,\leq \,\frac{6\sqrt{5}}{(5/2)^{k/2}} <\frac{1}{2}.\]
 Also, note that 
$\Re{((\tfrac{1}{2}+\ii t)+1)^{-k}} \leq \frac{1}{2}$
for all $k\geq 1$ and $t\geq \frac{1}{2}\sqrt{3}$. We conclude $(\Re E_{k}|\gamma)(\tau) > 0$ for all $\tau\in\mathcal{R}$ if $\lambda>1$ and $k\geq 11$.     
\end{proof}

These first two lemmata suffice to generalize \cref{prop:thm1}, i.e., to prove \cref{thm:1} for~$\lambda>1$:
\begin{prop}\label{prop:lambda>1}
For all $\lambda>1$ and $k\geq 3$ one has
\[ N_\lambda(E_k) \= \round\Bigl(\frac{k}{12}\Bigr)  \= N_\lambda(\GEis_k).\]
\end{prop}
\begin{proof}
We compute
\[(E_{k}|\gamma)(\rho) \= 
\begin{cases} 
 1 & k\equiv 0 \bmod 3 \\
1-\ii\sqrt{3} & k\equiv 1 \bmod 3 \\ 1+\ii\sqrt{3} & k\equiv 2 \bmod 3 \end{cases} 
\ \+\frac{1}{(-\frac{3}{2}+\frac{1}{2}\ii\sqrt{3})^k}
\+R_{k,\lambda}(\rho),\]
and similarly for $(E_{k}|\gamma)(\rho+1) $. Hence, using~\eqref{eq:rkest2}, we obtain
\begin{align}
|(E_k|\gamma)(\rho)-(E_k|\gamma)(\rho+1)|&\,<\,\frac{2}{3^{k/2}}\+\frac{12\sqrt{5}}{(5/2)^{k/2}} \,<\, \frac{30}{(5/2)^{k/2}}.
\end{align}
By \cref{lem:angle} and \cref{lem:R} this implies that
\[\bigl|\VOA_{\mathcal{L}}(E_k|\gamma)\+\VOA_{\mathcal{R}}(E_k|\gamma)\bigr|\,<\,\frac{30}{(5/2)^{k/2}}. \]
Moreover, by \cref{lem:C} 
\begin{align}
\VOA_\mathcal{C}(\widehat{E}_k|\gamma) &\,\approx \,  
\begin{cases} 
-\frac{1}{12} & k \equiv 1 \bmod 12 \\
-\frac{1}{4} & k \equiv 3 \bmod 12 \\
-\frac{5}{12} & k \equiv 5 \bmod 12 \\
\frac{5}{12} & k \equiv 7 \bmod 12 \\
\frac{1}{4} & k \equiv 9 \bmod 12 \\
\frac{1}{12} & k \equiv 11 \bmod 12.
\end{cases}
\end{align}
More precisely,
\begin{align}\Bigl|\VOA_{\mathcal{C}}(\widehat{E_k|\gamma}) \meno \Bigl(\round\Bigl(\frac{k}{12}\Bigr)-\frac{k}{12}\Bigr)\Bigr| \,<\,  \frac{30}{(5/2)^{k/2}}. \end{align}
Hence,
\begin{align} \label{eq:VOAonC}
\Bigl|\VOA_{\mathcal{C}}(E_k|\gamma) \meno \round\Bigl(\frac{k}{12}\Bigr)\Bigr| \,<\, \frac{30}{(5/2)^{k/2}}.  \end{align}
As 
\[N_\lambda(E_k) \= \bigl(\VOA_{\mathcal{L}}\+\VOA_{\mathcal{C}}\+\VOA_{\mathcal{R}}\bigr)(E_k|\gamma)\]
is an integer, the statement follows for $k\geq 13$. By improving the error estimates for small $k$, one obtains the same result for all $k\geq 3$. 

Recall $\GEis_k$ is proportional to $E_k$ up to the summand $\frac{1}{(c\tau+d)^k}$ with $-\frac{c}{d} =\lambda>1$ (see~\eqref{eq:GkEkcomp}). Now, since $c^2+d^2\geq 5$, this term is taken care of in the error estimates, from which we conclude that  \cref{lem:C}, \cref{lem:R} and this result also hold for $\GEis_k$ if $\lambda>1$. 
\end{proof}

\subsection{The case \texorpdfstring{$\lambda=0$}{lambda is 0}}
\begin{lem}\label{lem:relambda0}
For $\lambda=0$ and $k\geq 11$, the value
\begin{itemize}
\item $\Re \widehat{E}_k|\gamma(\tau)$ is non-zero for $\tau \in \mathcal{C}$\emph{;}
\item $\Re E_k|\gamma(\tau)$ is non-zero for $\tau\in \mathcal{R}$.
\end{itemize}
Here, $\gamma$ is such that $\lambda(\gamma)=0$
\end{lem}
\begin{proof}
Note
\[ \Re \widehat{E_k|\gamma}(\theta)\= \frac{1}{(2\cos\tfrac{1}{2}\theta)^k} \+ \Re\widehat{R}_{k,0}(\theta).\]
Now, similar as to \cref{lem:I} we estimate for all $k\geq 9$
\begin{align}
|\Re\widehat{R}_{k,0}(\theta)| &\,\leq\, \frac{4\sqrt{10}}{5^{k/2}}\+\begin{cases} 2\cdot3^{-k/2} & \theta\in [\pi/3,\pi/2] \\ 2\cdot5^{-k/2} & \theta\in [\pi/2,2\pi/3].\end{cases}
\end{align}
Hence, by a similar argument as in \cref{lem:thm1}, we conclude that $\Re\widehat{E_{k}|\gamma}(\theta)\neq 0$ for $k\geq 11$. 

For the second statement, let $\tau \in \mathcal{R}$ and note that
\[  \Im \frac{1}{\tau^k}+\Im \frac{1}{(\tau-1)^k} \= 0
\]
and
$|\tau+1|^{-k} \, \leq \, {3^{-k/2}}$. Hence, we obtain
\[ \Im E_k|\gamma(\tau) \= 1 + \frac{1}{(\tau+1)^k}+R_{k,0}(\tau)
\,\geq\, 1- \frac{1}{3^{k/2}} -\frac{6\sqrt{5}}{(5/2)^{k/2}},
\]
which implies the second part for $k\geq 11$. 
\end{proof}

\begin{prop}\label{prop:thm10}
For all odd $k\geq 11$
\[ N_0(E_k) \= 
\begin{cases} 
\lfloor \frac{k}{12} \rfloor & \text{if } k\equiv 1,7,9 \bmod 12 \\[3pt]
\lceil \frac{k}{12} \rceil &  \text{if } k\equiv 3,5,11 \bmod 12.
\end{cases}
\]
\end{prop}
\begin{proof}
Let $\gamma$ be such that $\lambda(\gamma)=0$. We compute
\[(E_{k}|\gamma)(\rho+1) \= 
\begin{cases} 
 1 & k\equiv 0 \bmod 3 \\
1+\ii\sqrt{3} & k\equiv 1 \bmod 3 \\ 
1-\ii\sqrt{3} & k\equiv 2 \bmod 3, \end{cases}  
 +\frac{1}{(\frac{3}{2}+\frac{1}{2}\ii\sqrt{3})^k}+R_{k,0}(\rho+1).\]
Recall $E_k|\gamma(i\infty)=1$ and $\Re E_k|\gamma(\tau)$ is non-zero for $\tau\in \mathcal{R}$. Hence, 
\[ \VOA_{\mathcal{R}}(E_k|\gamma) \approx \begin{cases} 
0 & k\equiv 0 \bmod 3\\
-\frac{1}{6} & k \equiv 1 \mod 3\\
\frac{1}{6} & k\equiv 2 \mod 3,
\end{cases}
\]
where the error is bounded in absolute value by $\frac{6\sqrt{5}}{(5/2)^{k/2}}.$
Moreover, we compute
\[(E_{k}|\gamma)(\rho) \= -1 + 3\delta_{3\nmid k} \meno \frac{1}{(-\frac{3}{2}+\frac{1}{2}\ii\sqrt{3})^k}+R_{k,0}(\rho) ,\]
where it can be computed that 
\[ \sgn \Im E_k|\gamma(\rho) = (-1)^{(k-3)/6} \qquad \text{if } 3\mid k.\]
Hence,
\[ \VOA_{\mathcal{L}}(E_k|\gamma) \approx \begin{cases} 
\frac{1}{2} & k\equiv 3 \bmod 12\\
-\frac{1}{2} & k\equiv 9 \bmod 12\\
0& 3\nmid k,
\end{cases}
\]
where the error is bounded in absolute value by $\frac{12\sqrt{5}}{(5/2)^{k/2}}.$
Finally, we find
\[
\VOA_{\mathcal{C}}(\widehat{E}_k|\gamma) \,\approx\, \begin{cases}
    \frac{1}{12} & k\equiv 1 \bmod 12 \\
    \frac{1}{4} & k\equiv 3 \bmod 12 \\
    \frac{5}{12} & k\equiv 5 \bmod 12 \\
    -\frac{5}{12} & k\equiv 7 \bmod 12 \\
     -\frac{1}{4} & k\equiv 9 \bmod 12 \\
      -\frac{1}{12} & k\equiv 11 \bmod 12,
\end{cases}
\]
where the error is bounded in absolute value by $\frac{12\sqrt{5}}{(5/2)^{k/2}}.$ Hence, by~\eqref{eq:VOA1} for $f=E_k|\gamma$ we obtain the desired result (note $N_\infty(f|\gamma)=N_\lambda(f)$). For $k\leq 11$ the same result can be obtained by more careful error estimates. 
\end{proof}

\subsection{The case \texorpdfstring{$0<\lambda<1$}{lambda is strictly between 0 and 1}}
In case $0<\lambda<1$, the variation of the argument on~$\mathcal{R}$ may be non-trivial, i.e., cannot directly be deduced from the non-vanishing of the real/imaginary part of~$E_k|\gamma$ on~$\mathcal{R}$. Instead, we split $\mathcal{R}$ in two segments, and show such a result for the real part on the one segment, and for the imaginary part on the other. 
\begin{lem}\label{lem:3.4}
For $0<\lambda \leq 1$ and $k\geq 19$, the value
$\Im E_k|\gamma(\tau)$ is non-zero for $\tau=\rho+1+\ii t \in \mathcal{R}$ with $0<t<\frac{\pi}{2k}$. Moreover, its sign is given by $\sgn(\frac{1}{2}-\lambda) (-1)^{\frac{k-1}{2}}$.
\end{lem}
\begin{proof}
For $\tau \in \mathcal{R}$, we have 
\[
\Im {E}_{k}|\gamma(\tau)
\= \Im\frac{1}{(\tau+2)^k} \+ \Im\frac{1}{(2\tau+1)^k} \+ \Im \frac{\sgn (\frac{1}{2}-\lambda)}{(-2\tau+1)^k} \+ \Im{\widetilde{R}}_{k,\lambda}(\tau)
\]
where ${\widetilde{R}}_{k,\lambda}$ consists of all terms with $m^2+n^2 \geq 10$.
Using a similar argument as in the proof of \cref{lem:R}, we find for $k \geq 15$
\begin{equation}
    |{\widetilde{R}}_{k,\lambda}| \,\leq\, 2 \!\!\!\sum_{\substack{m,n>0 \\ m^2 + n^2 \geq 10 \\ (m,n) = 1 }}\frac{1}{|m\tau+n|^k} \,\leq\, 2 \sum_{N\geq 10} \frac{N^{1/2}}{(\frac{1}{2}N)^{k/2}}
\,\leq\, \frac{16\sqrt{10}}{3\cdot 5^{k/2}} \qquad (\tau\in \mathcal{R}).
\end{equation}
Since both $|2 \tau + 1|^2$ and $|\tau + 2|^2$ are bounded from below by $7$ for $\tau \in \mathcal{R}$, we find that
\[ \left|\Im {E}_{k}|\gamma(\tau) -  \Im \frac{\sgn (\frac{1}{2}-\lambda)}{(1-2\tau)^k} \right| \,\leq\, \frac{19}{5^{k/2}}.\]
Clearly, for $\tau=\rho+1+\ii t$,
\[\Im \frac{\sgn(\frac{1}{2}-\lambda)}{(1-2 \tau)^k} \= \frac{\sgn(\frac{1}{2}-\lambda) (-1)^{\frac{k-1}{2}}}{(\sqrt{3}+2t)^k} \]
and this quantity is bounded in absolute value from below by $1/(\sqrt{3}+\frac{\pi}{k})^k$ for $0 < t < \frac{\pi}{2 k}$.
By the reverse triangle inequality, we conclude that
\[ 
\bigl|\Im {E}_{k}|\gamma(\tau) \bigr|\,\geq \,\frac{1}{(\sqrt{3}+\frac{\pi}{k})^k} - \frac{19}{5^{k/2}} .
\]
If $k \geq 19$, the right-hand side is positive and therefore $\Im {E}_{k}|\gamma(\tau)$ is non-zero for $0 < t < \frac{\pi}{2k}$.
\end{proof}

\begin{lem}\label{lem:relambda<1}
For $0<\lambda \leq 1$ and $k\geq 9$, the value
$\Re E_k|\gamma(\tau)$ is non-zero for $\tau=\rho+1+\ii t \in \mathcal{R}$ with $t>\frac{\pi}{2k}$. 
\end{lem}
\begin{proof}
Using~\eqref{eq:Ekhatgamma} we find the approximation
\[
\Re E_k|\gamma(\tau) \= 1 \+ 2 \,\Re\frac{1}{\tau^k} \+ \Re\frac{1}{(\tau+1)^k} \+ \Re R_{k,\lambda}(\tau) \]
for $\tau \in \mathcal{R}$.
As in the proof of \cref{lem:R}, we have the estimate
\[|R_{k,\lambda}(\tau)| \,\leq\, \frac{6 \sqrt{5}}{(5/2)^{k/2}}.\]
From the bounds $|\tau+1|^2 \geq 3$ and  $|\tau|^2 \geq \frac{1}{4} + (\frac{1}{2}\sqrt{3} + \frac{\pi}{2k})^2$ we get the following lower bound for the real part of~$E_k |\gamma$:
\begin{equation}
\bigl|\Re E_k | \gamma(\tau) \bigr| \,\geq\, 1 \meno \frac{6 \sqrt{5}}{(5/2)^{k/2}} \meno \frac{1}{3^{k/2}} \meno \frac{2}{(\frac{1}{4}+(\frac{1}{2}\sqrt{3} \+ \frac{\pi}{2 k})^2)^{k/2}}
\end{equation}
for $t>\frac{\pi}{2 k}$. It can be checked that the left-hand side is positive if $k \geq 9$. We conclude that $\Re E_k | \gamma(\tau) $ is non-zero for $t > \frac{\pi}{2k}$.
\end{proof}

\begin{prop}\label{prop:thm1b}
For odd $k\geq 19$ and $0<\lambda\leq 1$ the value $N_\lambda(E_k)$ is as in \cref{tab:1}. Moreover, further assuming $\lambda\neq 1$ we obtain $N_\lambda(\GEis_k)=N_\lambda(E_k)$.
\end{prop}
\begin{proof}
First, assume $\tfrac{1}{2} < \lambda \leq  1$. 
As before, we find
\[(E_{k}|\gamma)(\rho) \= \begin{cases} -1+\ii\sqrt{3} & k\equiv 1 \bmod 3 \\ -1-\ii\sqrt{3} & k\equiv 2 \bmod 3 \\ -1 & k\equiv 0 \bmod 3\end{cases} +\frac{1}{(-\frac{3}{2}+\frac{1}{2}\ii\sqrt{3})^k}+R_{k,\lambda}(\rho).\]
Recall $E_k|\gamma(\ii\infty) = 1.$ By a careful analysis as before, using \cref{lem:angle} and the fact that for $\tau \in \mathcal{L}$
\[ E_k|\gamma(\tau) \= E_k|\tilde{\gamma}(\tau+1)\]
for some $\tilde\gamma \in \sltwoz$ with $\lambda(\tilde\gamma) = \lambda(\gamma)+1$ and $\tau+1\in\mathcal{R}$, we find
\begin{align}
\VOA_{\mathcal{L}}(E_k|\gamma) \,\approx\, \begin{cases}0 & k\equiv 0 \bmod 3 \\ -\frac{1}{6} & k\equiv 1 \bmod 3 \\ \frac{1}{6} & k\equiv 2 \bmod 3 ,\end{cases}
\end{align}
where the error is bounded in absolute value by $\frac{6\sqrt{5}}{(5/2)^{k/2}}.$

Next, we have
\[(E_{k}|\gamma)(\rho+1) \= 1-3\delta_{3\nmid k} + \frac{1}{(\frac{3}{2}+\frac{1}{2}\ii\sqrt{3})^k}+R_{k,\lambda}(\rho+1).\]
Write $\alpha_k=\alpha=\rho+1+\frac{\pi}{2k}\ii$. Then, by the estimates in \cref{lem:3.4} and~\ref{lem:relambda<1} for $k\geq 19$ we find
\[ 
|\Im(E_k|\gamma)(\alpha)| \leq \frac{1}{3^{k/2}} , \qquad\Re(E_k|\gamma)(\alpha) \geq \frac{4}{9}. 
\]
Hence, by \cref{lem:angle},~\ref{lem:3.4} and~\ref{lem:relambda<1}, we obtain
\begin{align}
\VOA_R(E_k|\gamma) &\,\approx \, \begin{cases} 0 & 3\nmid k \\ -\frac{1}{2} & k\equiv 3 \bmod 12 \\ \frac{1}{2} & k\equiv 9 \bmod 12 \end{cases} 
\end{align}
with the error bounded in absolute value by $4\cdot {3}^{-k/2-2}$. Together with the variation of the argument on~$\mathcal{C}$, given by~\eqref{eq:VOAonC}, this gives the desired result for $N_\lambda(E_k|\gamma)$ for $\frac{1}{2}<\lambda\leq 1$ and $k\geq 19$. By improving the error estimates, we obtain the same result for $k\geq 3$.

The case $0 < \lambda \leq \tfrac{1}{2}$ goes analogously. The different outcome is caused by a sign change in $\VOA_{\mathcal{R}}(E_k|\gamma)$ due to a sign change in~$\Im E_k|\gamma$ (see \cref{lem:3.4}). 

Finally, not that the difference of~$E_k$ and $\GEis_k$ is contained in the remainder $R_{k,\lambda}$ if $\lambda \neq 1$. Hence, the same results hold for $\GEis_k$. 
\end{proof}

\subsection{Proof of~\texorpdfstring{\cref{thm:1}}{Theorem 1} and~\texorpdfstring{\cref{thm:3}}{Theorem 3}}
\begin{proof}[Proof of~\cref{thm:1} and~\cref{thm:3}] 
Observe that $N_\lambda(E_k)$ is already computed in many cases: $\lambda=\infty$ (\cref{prop:thm1}
), $\lambda>1$ (\cref{prop:lambda>1}), $0<\lambda\leq1$ (\cref{prop:thm1b}) and $\lambda=0$ (\cref{prop:thm10}). Moreover, for $\lambda \in \mathbb{P}^1(\q)\backslash\{0,\pm 1,\infty\}$ these results give the value of~$N_\lambda(\GEis_k)$. Hence, it would suffice to extend these results to negative values of~$\lambda$. 

Now, observe that all Fourier coefficients of~$i\GEis_k$ are purely real. Namely, in contrast to~$E_k\mspace{1mu}$, the constant term of~$\GEis_k$ vanishes. Hence, by \cref{lem:-lambda}, we conclude that
\[N_\lambda(\GEis_k) \= N_{-\lambda}(\GEis_k)\]
for all $\lambda\in \q$. By similar arguments as in the aforementioned propositions, we find that
\[ N_{-\lambda}(E_k)\= N_{-\lambda}(\GEis_k) \]
as long as $\lambda \in \mathbb{P}^1(\q)\backslash\{0,\pm 1,\infty\}$. 
\end{proof}

\begin{remark}
For $\lambda \in \{0,\pm 1,\infty\}$, we expect $N_\lambda(\GEis_k)$ equals the value which can be found in \cref{tab:2}. In this work, of these boundary cases only the case $\lambda=\infty$ is proven. We invite the reader to prove the other cases.
 \begin{table}[h!]
\arraycolsep=5pt\def\arraystretch{1.4} \begin{center}
 $\begin{array}{l|ccc}
k \, (12)
 &  \lambda=0 & \lambda= \pm 1 & \lambda=\infty\\\hline
1 & 0 & \color{blue} \left\lfloor\frac{k}{12}\right\rfloor & \color{blue}\left\lfloor \frac{k}{6} \right\rfloor\\ 
3 & 0 & \color{blue} \left\lfloor\frac{k}{12}\right\rfloor & \color{red}\left\lceil \frac{k}{6} \right\rceil\\ 
5 & 0 & \frac{k+1}{12} & \color{red}\left\lceil \frac{k}{6} \right\rceil\\ 
7 &0 & \frac{k-1}{12} & \color{blue}\left\lfloor \frac{k}{6} \right\rfloor\\ 
9  &0 & \color{red} \left\lceil\frac{k}{12}\right\rceil& \color{blue}\left\lfloor \frac{k}{6} \right\rfloor\\ 
11& 0 &\color{red} \left\lceil\frac{k}{12}\right\rceil & \color{red} \left\lceil \frac{k}{6} \right\rceil
 \end{array}$
 \end{center}\vspace{-8pt}
  \caption{The (special) values of~$N_\lambda(\GEis_k)$.}\label{tab:2}
 \end{table}
\end{remark}


\begin{thebibliography}{GKZ06}
\small 
\bibitem[Dim19]{Dim19}
Vesselin Dimitrov.
\newblock {\em A proof of the Schinzel--Zassenhaus conjecture on polynomials}.
\newblock { \href{https://arxiv.org/abs/1912.12545}{ArXiv e-prints:1912.12545}}, 27~pp., 2019.

\bibitem[EBS10]{ES10}
Abdelkrim El~Basraoui and Abdellah Sebbar.
\newblock {\em Zeros of the {E}isenstein series {$E_2$}}.
\newblock { Proc.\ Amer.\ Math.\ Soc.} \textbf{138}(7): 2289--2299, 2010.

\bibitem[GKZ06]{GKZ06}
Herbert Gangl, Masanobu Kaneko, and Don Zagier.
\newblock {\em Double zeta values and modular forms}.
\newblock In { Automorphic forms and zeta functions}: 71--106. World
  Sci.\ Publ., Hackensack, NJ, 2006.

\bibitem[Gek01]{Gek01}
Ernst-Ulrich Gekeler.
\newblock {\em Some observations on the arithmetic of {E}isenstein series for the
modular group {$\mathrm{SL}(2,\mathbb{Z})$}}.
\newblock  Arch. Math. \textbf{77}: 5--21, 2001.

\bibitem[GO22]{GO22}
Sanoli Gun and Joseph Oesterl\'{e}.
\newblock {\em Critical points of {E}isenstein series}.
\newblock {Mathematika} \textbf{68}(1): 259--298, 2022.

\bibitem[IJT14]{IJT14}
\"{O}zlem Imamoglu, Jonas Jermann, and \'{A}rp\'{a}d T\'{o}th.
\newblock {\em Estimates on the zeros of {$E_2$}}.
\newblock { Abh.\ Math.\ Semin.\ Univ.\ Hambg.} \textbf{84}(1): 123--138, 2014.

\bibitem[IR22]{IR22b}
Jan-Willem van Ittersum and Berend Ringeling.
\newblock {\em Critical points of modular forms}.
\newblock { \href{https://arxiv.org/abs/2204.00432}{ArXiv e-prints:
  2204.00432}}, 29 pp., 2022.

\bibitem[Noz08]{Noz08}
H. Nozaki.
\newblock {\em A separation property of the zeros of {E}isenstein series for
{$\mathrm{SL}(2,\mathbb{Z})$}}.
\newblock Bull. Lond. Math. Soc. \textbf{40}(1): 26--36, 2008.

\bibitem[RSD70]{RSD70}
Fanny K.~C.~Rankin and H.~Peter~F.~Swinnerton-Dyer.
\newblock {\em On the zeros of {E}isenstein series}.
\newblock {Bull.\ London Math.\ Soc.} \textbf{2}:169--170, 1970.

\bibitem[RVY17]{RVY17}
Sarah Reitzes, Polina Vulakh and Matthew P. Young.
\newblock {\em Zeros of certain combinations of {E}isenstein series}.
\newblock Mathematika {\textbf{63}} (2): 666--695, 2017.
	
\bibitem[Ran82]{Ran82}
Robert A. Rankin.
\newblock {\em The zeros of certain {P}oincar\'{e} series}.
\newblock Compositio Math. {\textbf{46}}(3), 255--272, 1982.

\bibitem[Whe23]{Whe23}
Campbell Wheeler.
\newblock {\em Modular $q$–difference equations and quantum invariants of
  hyperbolic three–manifolds}.
\newblock PhD thesis, Rheinischen Friedrich–Wilhelms–Universität Bonn,
  2023.

\bibitem[WY14]{WY14}
Rachael Wood and Matthew~P. Young.
\newblock {\em Zeros of the weight two {E}isenstein series}.
\newblock {J. Number Theory} \textbf{143}:320--333, 2014.

\bibitem[Zag20]{Zag20}
Don Zagier.
\newblock {\em Holomorphic quantum modular forms}.
\newblock \href{https://www.youtube.com/watch?v=2Rj_xh3UKrU}{Conference:
  Transfer operators in number theory and quantum chaos, Hausdorff Center for
  Mathematics, Bonn}, 2020.

\end{thebibliography}
\end{document}